\documentclass[11pt]{amsart}

\usepackage{amssymb,amsmath,amsthm}
\usepackage[usenames,dvipsnames,svgnames,table]{xcolor}
\usepackage[colorlinks=true, pdfstartview=FitV, linkcolor=blue, citecolor=blue, urlcolor=blue]{hyperref}

\usepackage[top=1.25in,bottom=1.5in,left=1in,right=1in,includehead]{geometry}
\usepackage{graphicx}
\usepackage{amssymb}
\usepackage{epstopdf}
\usepackage{lscape}
\usepackage[utf8]{inputenc}
\usepackage{tikz,caption}
\usepackage{enumitem}
\setlist[itemize]{leftmargin=2em}
\setlist[enumerate]{leftmargin=2em}
\usepackage{booktabs}
\usepackage{multirow}
\usepackage{mathtools}
\usepackage{caption}
\usepackage{subcaption}
\usepackage{paralist}
\definecolor{purple}{rgb}{0.56, 0.0, 1}

\newcommand{\qbinom}[2]{\begin{bmatrix} #1\\#2 \end{bmatrix}}

\newcommand{\qfactorial}[1]{\left[#1\right]!}
\newcommand{\qnumber}[1]{\left[#1\right]}
\newcommand{\qfalling}[2]{\left[#1\right]_{#2}}

\newcommand{\csf}[1]{X_{#1}({\bf x})}

\newcommand{\csft}[2]{X_{#1}({\bf x}, #2)}

\DeclareMathOperator{\LLT}{LLT}

\newcommand{\tcr}{\textcolor{red}}
\newcommand{\tcb}{\textcolor{blue}}

\newcommand{\qhit}[3]{H_{#1}^{#2}(#3)}
\newcommand{\qrook}[2]{R_{#1}(#2)}
\newcommand{\nqhit}[3]{\widetilde{H}_{#1}^{#2}(#3)}
\newcommand{\ac}[4]{a_{#1}^{#2}(#3,#4)}
\newcommand{\del}[3]{{#1}/^{#2}\, {#3}}

\DeclareMathOperator{\inv}{inv}
\DeclareMathOperator{\asc}{asc}

\DeclareMathOperator{\stat}{stat}
\DeclareMathOperator{\des}{des}
\DeclareMathOperator{\maj}{maj}
\DeclareMathOperator{\dstat}{stat_D}
\DeclareMathOperator{\cross}{cross}
\numberwithin{equation}{section}

\newtheorem{theorem}{Theorem}[section]
\newtheorem{lemma}[theorem]{Lemma}
\newtheorem{proposition}[theorem]{Proposition}
\newtheorem{corollary}[theorem]{Corollary}
\newtheorem{conjecture}[theorem]{Conjecture}
\newtheorem*{lemma*}{Lemma}

\newtheorem{definition}[theorem]{Definition}
\newtheorem{example}[theorem]{Example}
\newtheorem{remark}[theorem]{Remark}
\newtheorem{notation}[theorem]{Notation}

\newcommand{\n}{\mathsf{n}}
\newcommand{\e}{\mathsf{e}}

\newcommand{\thmref}[1]{\hyperref[#1]{Theorem~\ref*{#1}}}
\newcommand{\lemref}[1]{\hyperref[#1]{Lemma~\ref*{#1}}}
\newcommand{\propref}[1]{\hyperref[#1]{Proposition~\ref*{#1}}}
\newcommand{\corref}[1]{\hyperref[#1]{Corollary~\ref*{#1}}}
\newcommand{\conjref}[1]{\hyperref[#1]{Conjecture~\ref*{#1}}}
\newcommand{\defref}[1]{\hyperref[#1]{Definition~\ref*{#1}}}
\newcommand{\exref}[1]{\hyperref[#1]{Example~\ref*{#1}}}
\newcommand{\remref}[1]{\hyperref[#1]{Remark~\ref*{#1}}}
\newcommand{\secref}[1]{\hyperref[#1]{Section~\ref*{#1}}}
\newcommand{\figref}[1]{\hyperref[#1]{Figure~\ref*{#1}}}

\title[Chromatic symmetric functions of Dyck paths  and $q$-rook theory]{Chromatic symmetric functions \\of Dyck paths  and $q$-rook theory }

\author[Colmenarejo]{Laura Colmenarejo}
\address[L. Colmenarejo]{Department of Mathematics and Statistics, UMass Amherst, U.S.A}
\email{laura.colmenarejo.hernando@gmail.com}
\urladdr{https://sites.google.com/view/l-colmenarejo/home}

\author[Morales]{Alejandro H. Morales}
\address[A. H. Morales]{Department of Mathematics and Statistics, UMass Amherst, U.S.A}
\email{ahmorales@math.umass.edu}
\urladdr{https://people.math.umass.edu/~ahmorales/}

\author[Panova]{Greta Panova}
\address[G. Panova]{Department of Mathematics, University of Southern California, U.S.A.}
\email{gpanova@usc.edu}
\urladdr{https://sites.google.com/usc.edu/gpanova/home}

\thanks{L. Colmenarejo was partially supported by the AMS-Simons Travel Grant and by MTM2016-75024-P, A. H. Morales was partially supported by the NSF grant DMS-1855536, and G. Panova was partially supported by the NSF grant DMS-1939717.}

\begin{document}

\begin{abstract}
The chromatic symmetric function (CSF) of Dyck paths of Stanley  and its Shareshian--Wachs $q$-analogue have important  connections to  Hessenberg varieties, diagonal harmonics and LLT polynomials. In the, so called, abelian case they are also curiously related to placements of non-attacking rooks by results of Stanley--Stembridge (1993) and Guay-Paquet (2013). For the $q$-analogue, these  results have been generalized by  Abreu--Nigro (2020) and Guay-Paquet (private communication), using  $q$-hit numbers.  Among our main results is a new proof of Guay-Paquet's elegant identity expressing the $q$-CSFs in a CSF basis with $q$-hit coefficients.  We further show its equivalence to the Abreu--Nigro identity expanding the $q$-CSF in the elementary symmetric functions. In the course of our work we establish that the $q$-hit numbers in these expansions differ from the originally assumed Garsia-Remmel $q$-hit numbers by certain powers of $q$.  We  prove new identities for these $q$-hit numbers, and establish connections between the three different variants.
\end{abstract}

\maketitle

\textsc{keywords}: chromatic symmetric functions, abelian Hessenberg varieties, Dyck paths, $q$-hit numbers, $q$-rook numbers.

\section{Introduction}
Ever since their introduction in 1995 in \cite{St1}, the chromatic symmetric functions have been this mysterious object combining the misleading simplicity of graphs with the powerful tools of symmetric functions. Graph colorings present some of the hardest problems in combinatorics\footnote{Informally, but also formally as an NP-complete problem.}, and nice formulas there qualify as miracles rather than general rules. It is thus even more appealing that the chromatic symmetric functions, and their $q$-generalizations, are a source of beautiful results and striking conjectures\footnote{Most notably the $e$-positivity Conjecture~\ref{conj: StaSteShWa} of Stembridge-Stanley~\cite{StSt}, refined further by Shareshian-Wachs~\cite{ShW2}.}. 
The chromatic symmetric functions have found significant connections beyond combinatorics --  to \emph{Hessenberg varieties}~\cite{ShW2}, \emph{diagonal harmonics}~\cite{CM}, and {\em Macdonald polynomials}~\cite{AP,HW}.

In this paper we bring to light such an unusually nice combinatorial formula, relating the \emph{$q$-rook theory} which comes from generalizations of permutations and their inversions, and chromatic symmetric functions for Dyck paths of bounce two, aka abelian case.  We give an elementary proof of the strikingly elegant identity of Guay-Paquet (Theorem~\ref{thm:qhitCSFabelian}) which expresses the chromatic symmetric function for an arbitrary path given by partition $\lambda$ in terms of the chromatic symmetric functions for rectangles with coefficients the very combinatorial $q$-hit numbers. Along the way we establish numerous new identities for $q$-hit and $q$-rook numbers, give an elementary proof of Theorem~\ref{AN:generalLambda}, and pose many conjectures stemming from our findings. Our ultimate goal is to understand the chromatic symmetric functions with more relations and connections, which could lead not only to a proof of the $e$-positvity Conjecture~\ref{conj: StaSteShWa}, but also to a combinatorial interpretation of these coefficients. The technique of symmetry-breaking used in our proof of Theorem~\ref{thm:qhitCSFabelian}  could be extended beyond the abelian case as long as there is a suitable conjectured expression for the coefficients in the $e$-basis. 

\subsection{Definitions and main results}
Let $G$ be a graph with vertices $\{v_1,v_2,\ldots,v_n\}$ that are
totally ordered $v_1<v_2<\cdots <v_n$. In~\cite{St1}, Stanley defined the chromatic symmetric function (CSF) $\csf{G}$ of  $G$  as 
\begin{equation*}
\csf{G} = \sum_{\kappa: V\to \mathbb{P}, \text{ proper}} {\bf
  x}^{\kappa} =  \sum_{\kappa: V\to \mathbb{P}, \text{ proper}} x_1^{\#\kappa^{-1}(1)} x_2^{\#\kappa^{-1}(2)}\cdots, 
\end{equation*}
where $\mathbb{P}=\{1,2,3,\ldots\}$, ${\bf x}=(x_1,x_2,\ldots)$, and
the sum is over the proper colorings of the vertices of $G$.

Stanley and Stembridge~\cite{StSt} conjectured that the chromatic symmetric
functions  expand with positive coefficients in the basis $\{e_{\mu}\}$
of elementary symmetric functions for the graphs coming from Dyck paths in the following way. Given a Dyck path $d$ from $(0,0)$ to $(n,n)$, let $G(d)$ be the graph with vertices $\{1\ldots n\}$ and edges  $(i,j)$, $i<j$ if and only if the cell
$(i,j)$ is below the path $d$ (see Figure~\ref{fig: path to graph}). These are also the  incomparability graphs of {\em unit interval orders} or graphs obtained from {\em Hessenberg sequences}.  

\begin{figure}
    \centering
    \includegraphics{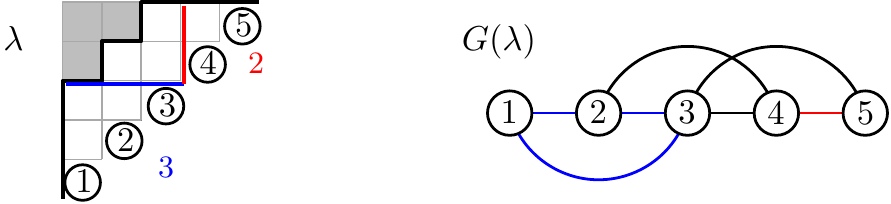}
    \caption{The Dyck path $d= \n^{\tcb{3}} \e\n\e\n\e \e^{\tcr{2}}$ associated to the partition $\lambda=(2,1) \subset 2 \times 3$ and the corresponding graph $G(\lambda)$. Each cell below the path $d$ corresponds to an edge of the graph. An example of a proper coloring would be $\kappa(1)=2, \kappa(2)=3,\kappa(3)=1,\kappa(4)=2,\kappa(5)=3$ which has $\asc(\kappa) = 4$.}
    \label{fig: path to graph}
\end{figure}

 Shareshian--Wachs~\cite{ShW2} introduced a quasisymmetric version of $\csf{G}$ defined by 
\begin{align*}
\csft{G}{q} = \sum_{\kappa: V\to \mathbb{P}, \text{ proper}} q^{\asc(\kappa)}{\bf x}^{\kappa},
\end{align*}
where $\asc(\kappa)$ is the number of edges $\{v_i,v_j\}$ of $G$ with
$i<j$ and $\kappa(v_i)<\kappa(v_j)$. 

For the graphs $G(d)$ coming from Dyck paths, the quasisymmetric function $\csft{G(d)}{q}$ is
actually symmetric and Shareshian--Wachs gave a refinement of the
Stanley--Stembridge conjecture for this Catalan family of graphs.

\begin{conjecture}[Stanley--Stembridge, Shareshian--Wachs] \label{conj: StaSteShWa}
Let $d$ be a Dyck path. Then the coefficients of $\csft{G(d)}{q}$ in
the elementary basis are in $\mathbb{N}[q]$.
\end{conjecture}

This conjecture has been verified independently and by different techniques by Cho--Huh~\cite{ChoHuh}, Harada--Precup~\cite{HP},  and Abreu--Nigro~\cite{AN} for the case of so-called {\em abelian} Dyck paths (corresponding to abelian Hessenberg varieties), which are defined as Dyck paths $d$ of from $(0,0)$ to $(m+n,m+n)$ of the form $\n^m w(\lambda) \e^n$ where $w(\lambda)$ is the encoding
in north ($\n$) and east ($\e$) steps of the partition $\lambda \subset n\times m$ (see
Figure~\ref{fig:hitexA}). We denote the associated graph by $G(\lambda)$ and the chromatic symmetric function by $\csft{\lambda}{q} := \csft{G(\lambda)}{q}$.

The symmetric functions $\csft{\lambda}{q}$ corresponding to abelian Dyck paths  are deeply related to the  {\em $q$-rook theory} of Garsia--Remmel~\cite{GR} as we illustrate with the next two identities that use the following notation 
\[
[n]_k = [n][n-1]\cdots[n-k+1],\quad  [n]!=[n]_n, \quad \qbinom{n}{k} =\frac{[n]_k}{[k]!},
\]
where $[x]=(1-q^x)/(1-q)$.

{\bf We define {\em $q$-hit numbers} of {\em rectangular boards}} of size $n\times m$ that we denote as $\qhit{j}{m,n}{\lambda}$ by a change of basis equation \eqref{eq: hit rook change of basis} involving the \emph{Garsia--Remmel $q$-rook numbers}. These $q$-hit numbers are polynomials in $q$, satisfying $\sum_{j=0}^n \qhit{j}{m,n}{\lambda}=[m]_n$, and at $q=1$ give the number of placements of $n$ non-attacking rooks in an $n \times m$ board ($n\leq m$) with $j$ rooks in the board of $\lambda$. We show that these $q$-hit numbers $\qhit{j}{m,n}{\lambda}$ are symmetric  polynomials in $\mathbb{N}[q]$ and are realized by a statistic defined by Dworkin \cite{D} (see Theorem~\ref{thm: qhit statistic rectangular board}). In the case of a square board $m=n$, these $q$-hit numbers $\qhit{j}{n}{\lambda}:=\qhit{j}{n,n}{\lambda}$ are up to a power of $q$ equal to the \emph{Garsia--Remmel $q$-hit numbers} (Proposition~\ref{prop: difference GR and our qhit}) which are symmetric unimodal polynomials in $\mathbb{N}[q]$ realized by different statistics by Haglund and Dworkin  (see \cite{HR}).

Abreu--Nigro gave an expansion of $\csft{\lambda}{q}$ in the elementary basis in terms of $q$-hit numbers of square boards. This result is a $q$-analogue of a special case of a result of Stanley--Stembridge~\cite[Thm. 4.3]{StSt}. 

\begin{theorem}[Abreu--Nigro~\cite{AN}]\label{AN:generalLambda}
Let $\lambda$ be partition inside an $n\times m$ board with $\ell(\lambda) = k \leq \lambda_1$. Then
\begin{align*}
\csft{\lambda}{q} &=
\qfactorial{k}\qhit{k}{m+n-k}{\lambda}\cdot e_{m+n-k,k} + \sum_{j=0}^{k-1}
q^j \qfactorial{j}\qnumber{m+n-2j} \qhit{j}{m+n-j-1}{\lambda} \cdot e_{m+n-j,j}.
\end{align*}
\end{theorem}

\textbf{Our first main result is an elementary proof of an unpublished identity of Guay-Paquet}\footnote{Private communication~\cite{MGP_LR}.}\footnote{This identity was independently found by Lee and Soh \cite[Thm. 24]{LS} after this article was posted.} in Section~\ref{sec: pf of MGP} that expands $\csft{\lambda}{q}$ in terms of
chromatic symmetric functions for rectangular shape with  coefficients given by the {\em $q$-hit numbers} of  {\em rectangular boards} defined above. 

This result appears as a $q$-analogue of a special case of \cite[Prop. 4.1 (iv)]{MGP}. 

\begin{theorem}[Guay-Paquet~\cite{MGP_LR}] \label{thm:qhitCSFabelian}
Let $\lambda$ be partition inside an $n\times m$ board ($n\leq m$). Then 
\[
\csft{\lambda}{q} =\dfrac{1}{\qfalling{m}{n}} \sum_{j=0}^{n}
\qhit{j}{m,n}{\lambda} \cdot \csft{m^j}{q}.
\]
\end{theorem}

\textbf{Our second result is a direct elementary proof} of Theorem~\ref{AN:generalLambda} following from Theorem~\ref{thm:qhitCSFabelian} and using  \textbf{our third result} which is the definition and properties of the  $q$-hit numbers $\qhit{*}{*,*}{\lambda}$, including a deletion contraction relation, that we derive in Sections~\ref{sec:new_qrooks} and in the Appendix. 

\subsection{Old and new methods}
The original proofs of the two statements above  use a linear relation satisfied by $\csft{G(d)}{q}$ called the {\em modular relation}~\cite{AN,AS,MGP}. Our proof of Theorem~\ref{thm:qhitCSFabelian} uses a simple inductive approach on both the size $m+n$ of the graph and the number $M$ of variables (see Lemma~\ref{lem:X_recursion}). Our approach ignores/breaks the symmetry of the chromatic symmetric function by splitting the function as a polynomial in $x_k$, whose coefficients are polynomials in $x_1,\ldots,x_{k-1}$. The ultimate identities are derived from identities of the coefficients; the $q$-hit numbers. Such an approach could work in a more general setting if the coefficients in the expansion have some recursive combinatorial structure. Moreover, following the recursion it could be extended to a bijection, similar to RSK. The bottlenecks in this approach are the necessary new $q$-hit identities, which we derive after extensive use of generating functions\footnote{A fully combinatorial/bijective proof would be highly desirable and could completely unravel the combinatorics for CSFs in the abelian case. See Section~\ref{sec: final remarks}.}.  The derivation of Theorem~\ref{AN:generalLambda} follows from other $q$-hit identities, which can be proven also using deletion-contraction on $q$-hit and $q$-rook numbers. Note that deletion-contraction on the classical CSFs itself is not directly applicable due to the inhomogeneity of the relation.   

Along the way we prove new $q$-hit identities (Section~\ref{sec:new_qrooks} and Appendix) and unravel a mystery on different combinatorial statistics leading to different kinds of $q$-hit numbers (see Section~\ref{subsec:tale} and the Appendix) that have been mixed up in the literature. In particular, we establish new relations of {\em $q$-rook numbers} and {\em $q$-hit numbers} (Lemmas~\ref{lem:qrook_linear_1},~\ref{lem:qrook_linear_2},~\ref{lem:qrook_identity1}, and~\ref{prop:qhit-relations}) that develop further the $q$-rook theory of rectangular boards~\cite{LM17}.

As a Corollary to the fact that Theorems~\ref{AN:generalLambda} and ~\ref{thm:qhitCSFabelian} are in essence linear relation between chromatic symmetric functions, we establish that the same linear relation holds of the \textbf{unicellular LLT polynomials}, see Section~\ref{subsec:unicellular}.

\subsection{Organization}

In Sections~\ref{sec: background} and~\ref{sec:new_qrooks} we give the definitions of $q$-hit numbers and prove the necessary identities used later on. Our elementary proof of Theorem~\ref{thm:qhitCSFabelian} is in Section~\ref{sec: pf of MGP}, and the proof of Theorem~\ref{AN:generalLambda} is in Section~\ref{sec: MGP to AN}. In Section~\ref{sec: applications and variations} we discuss variations on these problems, expansions in other bases like CSFs for staircase shapes, applications to LLT polynomials, and some conjectures.

 In Appendices~\ref{app: Dworkin}, \ref{app: relation}, and~\ref{app: deletion-contraction} we present the Garsia--Remmel $q$-hit numbers and their relation to the $q$-hit numbers appearing in Theorems~\ref{AN:generalLambda},\ref{thm:qhitCSFabelian}, and the deletion-contraction relations for each variant.  Appendices~\ref{app: proof statistic},~\ref{app: symmetry qhit} have the proofs of Theorem~\ref{thm: qhit statistic rectangular board} and the symmetry of the $q$-hit numbers, respectively.

\section*{Acknowledgements}
We thank Mathieu Guay-Paquet for generously sharing the notes~\cite{MGP_LR} with 
Theorem~\ref{thm:qhitCSFabelian} as well as Alex Abreu, Per Alexandersson, Sergi Elizalde, Jim Haglund, Philippe Nadeau, Antonio Nigro, Alexei Oblomkov, Franco Saliola, Bruce Sagan, John Shareshian, and Michelle Wachs for insightful discussions. We also thank the anonymous referees for the comments and suggestions. This work was facilitated by computer experiments using Sage~\cite{sagemath} and its algebraic combinatorics features
developed by the Sage-Combinat community~\cite{Sage-Combinat}. 

\section{Background on $q$-rook theory}\label{sec: background}
For the rest of the paper, we assume $m$ and $n$ are non-negative integers with $m\geq n$. 

\subsection{$q$-rook numbers}

Rook placements are a generalizations of permutation diagrams, and their $q$-analogues keep track of the number of inversions. We now summarize important definitions and properties used later in relation to the chromatic symmetric functions. In the Appendix we include the proofs and further properties. 

\begin{definition}[$q$-rook numbers~\cite{GR}]\label{def: GR q-rook numbers}
Given a partition
$\lambda=(\lambda_1,\lambda_2,\ldots,\lambda_{\ell})$ the \emph{Garsia-Remmel $q$-rook numbers}
 are defined as 
 \[
 \qrook{k}{\lambda} = \sum_p q^{\inv(p)},
 \]
where the sum is over all placements $p$ of $k$ non-attacking rooks on $\lambda$ and
$\inv(p)$ is the number of cells of $\lambda$ that are not occupied by
a rook or directly
west or north of a rook (see Figure~\ref{fig:ex rook placement}). 
\end{definition}

\begin{proposition}[Garsia-Remmel~\cite{GR}]\label{prop:GR_rook_gen_fun}
Given a partition $\lambda=(\lambda_1,\ldots,\lambda_{\ell})$ we have that 
\begin{equation*} %\label{eq: def F}
F(x;\lambda):=\sum_{k=0}^{\ell} \qrook{k}{\lambda}[x]_{\ell-k}  = \prod_{i=1}^{\ell} [x+\lambda_{\ell-i+1}-i+1],
\end{equation*}
in particular $R_{\ell}(\lambda)=\prod_{i=1}^{\ell} [\lambda_{\ell-i+1} -i+1]$.
\end{proposition}

\subsection{$q$-hit numbers}

The $q$-hit numbers are defined in terms of the $q$-rook numbers by a change of basis. Let $(a;q)_k=\prod_{i=0}^{k-1}(1-aq^i)$ denote the {\em $q$-Pochhammer symbol}.

\begin{definition}[{\cite[Def. 3.1, Prop. 3.5]{LM17}}]\label{lem: hit in terms of rs}
For $\lambda$ inside an $n\times m$ board, we define the \emph{$q$-hit polynomial} of $\lambda$ by 
\begin{equation} \label{eq: hit rook change of basis}
P(x;\lambda) = \sum_{i=0}^n \qhit{i}{m,n}{\lambda} x^i := \frac{q^{-|\lambda|}}{\qfactorial{m-n}}\sum_{i=0}^n \qrook{i}{\lambda} \qfactorial{m-i} (-1)^i q^{mi-\binom{i}{2}} (x;q)_i,
\end{equation}
where the coefficients $\qhit{i}{m,n}{\lambda}$ are the \emph{$q$-hit numbers} associated to $\lambda$. Equivalently, we have that for every $k$
\begin{align} \label{eq: hit in terms of rs}
\qhit{k}{m,n}{\lambda} &= \dfrac{q^{\binom{k}{2}-|\lambda|}}{\qfactorial{m-n}}\sum_{i=k}^n \qrook{i}{\lambda} \qfactorial{m-i} \qbinom{i}{k} (-1)^{i+k} q^{mi-\binom{i}{2}},\\ 
\intertext{and}
\qrook{k}{\lambda} &= q^{|\lambda|-mk}\frac{[m-n]!}{[m-k]!}\sum_{i=k}^n \qhit{i}{m,n}{\lambda} \qbinom{i}{k}_{q^{-1}}. \label{eq: r in terms of hits}
\end{align}
\end{definition}

\begin{notation}
For square boards with $n=m$, we denote the $q$-hit number by $\qhit{j}{m}{\lambda}$. 
\end{notation}

\begin{remark} \label{rem: difference qhits}
For the case $n=m$, Garsia--Remmel defined $q$-hit numbers $\nqhit{k}{n}{\lambda}$ by the relation
\begin{equation} \label{eq: def GR qhit}
\sum_{i=0}^n \nqhit{i}{n}{\lambda} x^i = \sum_{i=0}^n \qrook{i}{\lambda} [n-i]! \prod_{k=n-i+1}^n (x-q^k).
\end{equation}
One can show that the Garsia--Remmel $q$-hit numbers and our $q$-hit numbers differ by a power of $q$ (see Proposition~\ref{prop: difference GR and our qhit}).
\end{remark}

The $q$-hit numbers satisfy the following deletion-contraction relation that is proved in Appendix~\ref{app: deletion-contraction}. Given a shape $\lambda$ and a corner cell $e$ in $\lambda$, $\lambda \backslash e$  denotes the shape obtained after deleting the cell $e$ in $\lambda$, and $\lambda /e$ denotes the shape obtained after deleting in $\lambda$ the row and column containing $e$. See Figure~\ref{fig:ex deletion-contraction} for an example. 
\begin{figure}[ht]
\centering
\includegraphics{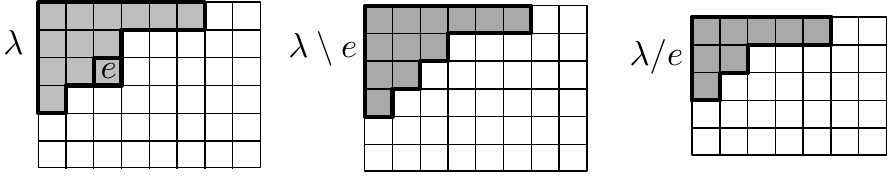}
\caption{Example of the deletion and contraction of the board of a partition $\lambda$.}
\label{fig:ex deletion-contraction}
\end{figure}

\begin{lemma} \label{lem: deletion/contration}
We have the following deletion-contraction relation:
\begin{equation*} 
\qhit{j}{m,n}{\lambda}  = \qhit{j}{m,n}{\lambda\backslash e}
+ q^{|\lambda/e|-|\lambda|+j+m-1}\left( \qhit{j-1}{m-1,n-1}{\lambda/e} - q \qhit{j}{m-1,n-1}{\lambda/e}\right).
\end{equation*}
\end{lemma}

Guay-Paquet~\cite{MGP_LR} defined the rectangular $q$-hit numbers using a statistic similar to Dworkin's statistic \cite{D} for the Garsia--Remmel $q$-hit numbers and we present this definition next, illustrated in Figure~\ref{fig:ex hit number}.

\begin{definition}[Statistic for the $q$-hit numbers]\label{def: our stat qhit}
Let $\lambda$ be a partition inside an $n\times m$ board. Given a placement $p$ of $n$ non-attacking rooks on an $n\times m$ board, with exactly $j$ rooks inside $\lambda$, let $\stat(p)$ be the number of cells $c$ in the board such that:
\begin{compactitem}
\item[(i)] there is no rook in $c$, 
\item[(ii)] there is no rook above $c$ on the same column, and either, 
\item[(iii)] if  $c$ is in $\lambda$ then the rook on the same row of $c$ is in $\lambda$ and to the right of $c$ or  
\item[(iv)] if $c$ is not in $\lambda$ then the rook on same row of $c$ is either in $\lambda$ or to the right of $c$. 
\end{compactitem}
\end{definition}

\begin{remark} \label{rem: rook cancellation in cylinder}
Intuitively, this statistic $\stat(p)$ counts the number of remaining cells in the $n\times m$ board after: wrapping this board on a vertical cylinder and each rook of $p$ cancels the cells south in its column and the cells east in its row until the border of $\lambda$.
\end{remark}

\begin{theorem} \label{thm: qhit statistic rectangular board}
Let $\lambda$ be a partition inside an $n\times m$ board and $j=0,\ldots,n$ then
\begin{equation} \label{eq: qhit statistic rectangular board}
\qhit{j}{m,n}{\lambda}= \sum_p q^{\stat(p)}, 
\end{equation}
where the sum is over all placements $p$ of $n$ non-attacking rooks on an $n\times m$ board, with exactly $j$ rooks inside $\lambda$.
\end{theorem}

The proof for Theorem \ref{thm: qhit statistic rectangular board} is given in Appendix \ref{app: proof statistic}.

\begin{remark}\label{remark: different qhits}
Note that the Garsia--Remmel $q$-hit numbers have a very similar description in \cite{HR} (attributed to Dworkin) using a different attacking rule for the rooks. Our proof of Theorem~\ref{thm: qhit statistic rectangular board} in Appendix~\ref{app: Dworkin} follows by reducing to the case of the Garsia--Remmel $q$-hit numbers. See also Section~\ref{subsec:tale} for more details.
\end{remark}

Moreover, for each partition $\lambda$, the statistic $\stat(\cdot)$ is \emph{Mahonian}.

\begin{corollary} \label{cor: sum of hits}
Let $\lambda$ be a partition inside an $n\times m$ board, then
\[
\sum_{j=0}^n \qhit{j}{m,n}{\lambda} = [m]_n.
\]
\end{corollary}

\begin{proof}
Set $k=0$ in~\eqref{eq: r in terms of hits} and since $\qrook{0}{\lambda}=q^{|\lambda|}$, we obtain
\[
\sum_{j=0}^n \qhit{j}{m,n}{\lambda} = q^{-|\lambda|} \qrook{0}{\lambda} [m]_n = [m]_n. \qedhere
\]
\end{proof}

\begin{figure}[ht]
\centering
\begin{subfigure}[b]{0.22\textwidth}
\includegraphics{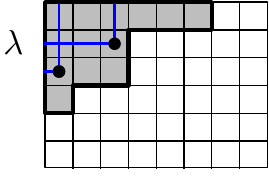}
\caption{}
\label{fig:ex rook placement}
\end{subfigure}
\begin{subfigure}[b]{0.22\textwidth}
\raisebox{5pt}{\includegraphics{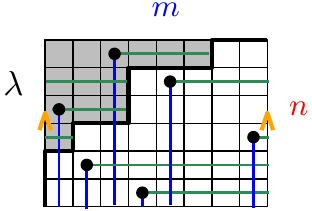}}
\caption{}
\label{fig:ex hit number}
\end{subfigure}
\begin{subfigure}[b]{0.22\textwidth}
\raisebox{5pt}{\includegraphics{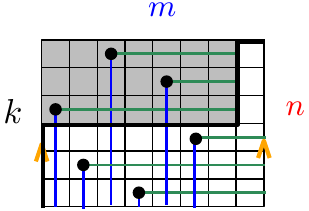}}
\caption{}
\label{fig:ex hit number rectangle k hits}
\end{subfigure}
\begin{subfigure}[b]{0.22\textwidth}
\raisebox{5pt}{\includegraphics{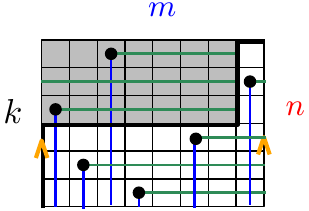}}
\caption{}
\label{fig:ex hit number rectangle k-1 hits}
\end{subfigure}

\caption{Example of the statistics of (A) a $q$-rook number and (B) a $q$-hit number. (C)(D) Examples of the cases of $q$-hit numbers of a rectangle $(m-1)^k \subset n\times m$ for Proposition~\ref{prop: hits small rect in rect}.}
\end{figure}

\begin{example} \label{ex: qrook and our qhit}
Consider the partition $\lambda=(6,3,3,1)$ inside a $6\times 8$ board. In Figures~\ref{fig:ex rook placement} and \ref{fig:ex hit number}, we present an example of a placement $p$ of two rooks on $\lambda$ with $\inv(p)=7$ and  an example of a placement $p'$ of six rooks on the  $6\times 8$ board with two hits on $\lambda$ and $\stat(p')=13$, respectively.
\end{example}

We finish this section with some results for $q$-hit numbers. 
The next two results show the relation between the $q$-hit numbers when we change the dimensions of the board.

\begin{lemma}\label{lem:q-hit for m,m}
Let $\lambda$ be a partition inside an $n\times m$ board. Then
\[
\qhit{j}{m,n}{\lambda} = \frac{1}{[m-n]!}\qhit{j}{m,m}{\lambda}.
\]
\end{lemma}

\begin{proof}
Since $\qrook{i}{\lambda}=0$ for $i=n+1,\ldots,m$ then~\eqref{eq: hit in terms of rs} becomes,
$$
\qhit{j}{m,n}{\lambda} = \dfrac{q^{\binom{j}{2}-|\lambda|}}{\qfactorial{m-n}}\sum_{i=j}^m \qrook{i}{\lambda} \qfactorial{m-i} \qbinom{i}{j} (-1)^{i+j} q^{mi-\binom{i}{2}} = \frac{1}{[m-n]!} \qhit{j}{m,m}{\lambda}.$$
\end{proof}

\begin{lemma} \label{lemma:remove column}
Let $\lambda$ be a partition inside an $(n-1)\times m$ board. Then
\begin{equation*} %\label{eq:greta-relation}
    \qhit{j}{m,n}{\lambda} = \qnumber{m+1-n} \qhit{j}{m,n-1}{\lambda}.
\end{equation*}
\end{lemma}

\begin{proof}
We apply~\eqref{eq: hit in terms of rs} %~\eqref{eq: hit rook change of basis} 
to $\qhit{j}{m,n}{\lambda}$

and use the fact  $\qrook{n}{\lambda}=0$ since $\lambda \subset (n-1)\times m$ to obtain
\begin{align*}
    \qhit{j}{m,n}{\lambda} &= 
     \qnumber{m+1-n}\dfrac{q^{\binom{j}{2}-|\lambda|}}{\qfactorial{m-n+1}}\sum_{i=j}^{n-1} \qrook{i}{\lambda} \qfactorial{m-i} \qbinom{i}{j} (-1)^{i+j} q^{mi-\binom{i}{2}} \\
    &= \qnumber{m+1-n}\qhit{j}{m,n-1}{\lambda} ,
\end{align*}
where we used~\eqref{eq: hit rook change of basis} again for $\qhit{j}{m,n-1}{\lambda}$ to obtain the desired formula.
\end{proof}
Finally, we give formulas for $q$-rook numbers and $q$-hit numbers of rectangular shapes. 

\begin{proposition} 
\begin{equation} \label{eq: rooks rectangles}
\qrook{k}{a^b} = q^{(a-k)(b-k)}\dfrac{ \qfalling{a}{k} \qfalling{b}{k}}{\qfactorial{k}}.
\end{equation}
\end{proposition}

\begin{proof}
The result follows from the recurrence $\qrook{k}{a^b} = \qnumber{b}  \qrook{k-1}{(a-1)^{b-1}} +q^b \qrook{k}{(a-1)^b}$ or from the fact that up to a power of $q$ and $(q-1)^k$, this is the number of rank $k$ matrices of size $a\times b$ over a finite field with $q$ elements~\cite[Thm. 1]{H}, formulas for which can be found in~\cite[Sec. 1.7]{Morrison}.
\end{proof}

\begin{proposition} \label{prop:qhit_rectangles}
\begin{align}\label{eq:qhit_mj}
    \qhit{k}{N}{m^j}= q^{ (N-j-m+k)k} \qfalling{m}{k} \qfactorial{N-j}  \dfrac{ \qfalling{N-m}{j-k} \qfalling{j}{j-k}}{ \qfactorial{j-k}}.
\end{align}
\end{proposition}

\begin{proof}
We compute the $q$-hit number directly using the statistic in~\eqref{eq: qhit statistic rectangular board}. We claim that
\begin{align*}
    \qhit{k}{N}{m^j}= \qrook{k}{m^k}\cdot \qrook{j-k}{(N-m)^j}\cdot \qrook{N-j}{(N-j)^{N-j}}.
\end{align*}
To show this, let us denote by $B$ the $N\times N$ board. Then there are $j-k$ rows occupied by rooks right of the shape $m^j$. These rooks cancel the respective rows from the $m^j$ shape. The overall contribution to the $q$-hit number from the $k$ rooks on the remaining shape $m^k$ is $\qrook{k}{m^k}=\qfalling{m}{k}$. The overall contribution to the $q$-hit number from the $j-k$ rooks placed to the right of the shape $m^j$ is $\qrook{j-k}{(N-m)^j}$. The $j$ rooks placed on the first $j$ rows of the board $B$ cancel as many columns in the shape $N^{N-j}$ consisting of the last $N-j$ rows of the board $B$. 

By Proposition~\ref{prop:GR_rook_gen_fun}, the overall contribution to the $q$-hit number from placing the remaining $N-j$ rooks in the remaining shape $(N-j)^{N-j}$ is $\qrook{N-j}{(N-j)^{N-j}}=[N-j]!$. This proves the claim and the result follows by using the formula in~\eqref{eq: rooks rectangles} for $\qrook{k}{a^b}$.
\end{proof}

\begin{proposition} \label{prop: hits small rect in rect}
$$\qhit{r}{m,n}{(m-1)^k}= 
\begin{cases}
q^k[m-k] [m-1]_{n-1}& r=k, \\
[k] [m-1]_{n-1} & r=k-1, \\
0 & otherwise.
\end{cases} 
$$
\end{proposition}

\begin{proof}
Since every row of the $n\times m$ board has a rook and the last column has at most one rook then the rooks can only ``hit" the shape $(m-1)^k$ $r=k$ or $r=k-1$ times. 

When $r=k$, the first $k$ cells of the last column are not cancelled in any rook placement so they contribute to $\stat(\cdot)$. The contribution to the $q$-hit number from the $k$ rooks placed on the shape $(m-1)^k$ is $\qrook{k}{(m-1)^k}$. These rooks cancel $k$ columns in the shape of the last $n-k$ rows. Then the contribution to the $q$-hit number of placing $n-k$ rooks on the remaining shape $(m-k)^{n-k}$ is $\qrook{n-k}{(m-k)^{n-k}}$. See Figure~\ref{fig:ex hit number rectangle k hits}. Thus, 
\begin{multline*}
\qhit{k}{m,n}{(m-1)^k}  = q^k \qrook{k}{(m-1)^k}\cdot \qrook{n-k}{(m-k)^{n-k}} = q^k [m-1]_k[m-k]_{n-k}=  q^k[m-k] [m-1]_{n-1}.
\end{multline*}

When $r=k-1$, there is a rook on one of the first $k$ cells of the $m$th column which cancels all the cells of its corresponding row $i=1,\ldots,k$ and the cells below the rook in its column. This rook contributes $i-1$ to the statistic (the cells above the rook in the $m$th column). There contribution to the $q$-hit number from the $k-1$ rooks placed on the remaining shape $(m-1)^{k-1}$ (without row $i$) is $\qrook{k-1}{(m-1)^{k-1}}$. The $k$ rooks on the first $k$ rows cancel the cells of their columns in the shape of the last $n-k$ rows. Then the contribution to the $q$-hit number of placing $n-k$ rooks on the remaining shape $(m-k)^{n-k}$ is $\qrook{n-k}{(m-k)^{n-k}}$. See Figure~\ref{fig:ex hit number rectangle k-1 hits}. Summing over all $i=1,\ldots,k$ we obtain that
\[
\qhit{k-1}{m,n}{(m-1)^k} = [k]\qrook{k-1}{(m-1)^{k-1}} \cdot \qrook{n-k}{(m-k)^{n-k}} = [k] [m-1]_{k-1} [m-k]_{n-k} = [k] [m-1]_{n-1}.
\]
\end{proof}

\section{New $q$-rook and $q$-hit identities}\label{sec:new_qrooks}

The proofs of Theorems~\ref{AN:generalLambda} and~\ref{thm:qhitCSFabelian} rely on many new identities between $q$-hit numbers, and their equivalent $q$-rook versions. While the $q$-hit and $q$-rook identities are independent of the chromatic symmetric function theory, we first prove them using elementary combinatorial methods and univariate generating functions. 

For brevity, we will denote by $\del{\lambda}{c}{j}$ the partition obtained from $\lambda$ by removing its $j$th column, by $\del{\lambda}{r}{i}$ the partition obtained by removing its $i$th row, and by $\lambda/(i,j)$ the partition obtained from $\lambda$ by removing its $j$th column and its $i$th row. Moreover, we denote $\ell = \ell(\lambda)$ and the conjugate partition of $\lambda$ by $\lambda'$.

\subsection{Identities on $q$-rook numbers}

The following $q$-rook identities can be proven directly using the generating function identity in Proposition~\ref{prop:GR_rook_gen_fun} of Garsia-Remmel. We start with a simple algebraic identity. 

\begin{lemma}\label{lem:F_ratio}
 Given a partition $\lambda=(\lambda_1,\ldots,\lambda_{\ell})$,
    $$q^{\lambda_1} [x] \frac{ F(x-1;\lambda)}{F(x;\lambda)} =  [x-\ell+\lambda_1]  -  \sum_{j=1}^{\lambda_1} q^{\lambda_1-j}\prod_{t=1}^{\lambda_j'} \frac{[x+\lambda_{t}-1-\ell+t] }{[x+\lambda_{t}-\ell+t]}.$$
    \end{lemma}    

\begin{proof}
We use induction on $\ell(\lambda)$ and apply Proposition~\ref{prop:GR_rook_gen_fun}. For $\ell(\lambda)=1$, we have 
\begin{multline*}
    [x] \frac{ F(x-1;\lambda)}{F(x;\lambda)} =  [x-1] +  \sum_{j=1}^{\lambda_1} q^{-j}  -  q^{-j}  \frac{[x+\lambda_{1}-1] }{[x+\lambda_{1}]} = q^{-\lambda_1} \frac{[x+\lambda_1-1]}{[x+\lambda_1]}([x+\lambda_1]-[\lambda_1]).
\end{multline*}
Next, expanding the RHS of the above identity and doing standard manipulations gives
  \begin{multline*}
 [x-\ell+\lambda_1]  -  \sum_{j=1}^{\lambda_1} q^{\lambda_1-j}\prod_{t=1}^{\lambda_j'} \frac{[x+\lambda_{t}-1-\ell+t] }{[x+\lambda_{t}-\ell+t]} = \\
  = q^{\lambda_1-\lambda_2} \frac{[x+\lambda_1-\ell]}{[x+\lambda_1-\ell+1]} \left( [x+\lambda_2 -(\ell-1)] - \sum_{j=1}^{\lambda_2} q^{\lambda_2-j} \prod_{t=2}^{\lambda_j'} \frac{[x+\lambda_{t}-1-\ell+t] }{[x+\lambda_{t}-\ell+t]} \right).
\end{multline*}
By induction hypothesis the parenthetical on the RHS above is $q^{\lambda_2} [x]  F(x-1;\tilde{\lambda})/F(x;\tilde{\lambda})$ where $\tilde{\lambda} = (\lambda_2,\ldots,\lambda_\ell)$. Using  $\tilde{\lambda}_t = \lambda_{t+1}$ for the reindexing, we obtain the result. 
\end{proof}    

Next, we find the following $q$-analogue of the derivative of $F(x;\lambda)$.
\begin{lemma}\label{lem:DF}
We have the following formula for the $q$-analogue of the derivative of $F(x;\lambda)$:
$$DF(x;\lambda) :=\sum_{k=0}^{\ell} [k]\qrook{k}{\lambda} [x]_{\ell-k}=q^{\ell-x}( [x]F(x-1;\lambda)-[x-\ell]F(x;\lambda)).$$
\end{lemma}

\begin{proof}

Following the notation of Proposition~\ref{prop:GR_rook_gen_fun},  

$$
F(x;\lambda):= \sum_{k=0}^\ell \qrook{k}{\lambda} [x]_{\ell-k} =  \prod_{i=1}^\ell [x+\lambda_{\ell-i+1}-i+1].
$$
Note that 
\begin{align*}
[x]_{\ell-k} - [x-1]_{\ell-k} &= [x-1]_{\ell-k-1} \left( \frac{1-q^{x} - 1+q^{x-\ell+k}}{1-q} \right) = \frac{[x]_{\ell-k}}{[x]}  \frac{ q^{x-\ell} (q^{k}-q^{\ell})}{1-q}
\\ &= \frac{[x]_{\ell-k}}{[x]} q^{x-\ell} ( [\ell]-[k]).
\end{align*}

Apply this identity to each term in the following difference
\begin{align*}
    F(x;\lambda) - F(x-1;\lambda) = \sum_{k=0}^\ell \qrook{k}{\lambda} \frac{[x]_{\ell-k}}{[x]} q^{x-\ell} ( [\ell]-[k]),
    \end{align*}
so that
\begin{align*}
[x](F(x;\lambda) - F(x-1;\lambda)) = q^{x-\ell} [\ell] F(x;\lambda)- q^{x-\ell} DF(x;\lambda),
\end{align*}
where $DF(x;\lambda):=\sum_{k=0}^\ell [k]\qrook{k}{\lambda} [x]_{l-k}$. This gives an equation for $DF(x;\lambda)$ which we solve as
$$DF(x;\lambda) = -q^{\ell-x} ([x]-q^{x-\ell}[\ell])F(x;\lambda) + q^{\ell-x}[x]F(x-1;\lambda) =q^{\ell-x}( [x]F(x-1;\lambda)-[x-\ell]F(x;\lambda)),$$
giving us the desired formula.
\end{proof}
The following result shows the relationship between $q$-rook numbers for partitions obtained from deleting a column of $\lambda$.
\begin{lemma}\label{lem:qrook_linear_1}
For all $i$ fixed,
\[
\sum_j q^{m-j} \qrook{i}{\del{\lambda}{c}{j}}  =    \qrook{i}{\lambda} [m-i]    -  \qrook{i+1}{\lambda}  (q^m -q^{m-i-1}).
\]
\end{lemma}
\begin{proof}

Multiplying on both sides by $[x]_{\ell-i}$, the above claim is equivalent to the generating function identity:
\begin{multline*}
    \sum_j q^{m-j} F(x;\del{\lambda}{c}{j}) = \sum_i (\qrook{i}{\lambda} [m-i]    -  \qrook{i+1}{\lambda}  (q^m -q^{m-i-1}))[x]_{\ell-i}\\
    =\sum_i \qrook{i}{\lambda} \left( [m-i] [x]_{\ell-i} - (q^m-q^{m-i})[x]_{\ell-i+1} \right) = \sum_i \qrook{i}{\lambda} [x]_{\ell-i}[m+x-\ell] -q^m[x][x-1]_{\ell-i}\\
    = [m+x-\ell]F(x;\lambda) - q^m[x] F(x-1;\lambda),
\end{multline*}
where we use the observation that
\begin{multline*}
    [m-i] [x]_{\ell-i} - (q^m-q^{m-i})[x]_{\ell-i+1} = [x]_{\ell-i}\frac{(1-q^{m-i})  - (q^m-q^{m-i})(1-q^{x-\ell+i}) }{1-q} \\=  [x]_{\ell-i}\frac{1-q^{m-i}  - q^m+q^{m-i}+ q^{m+x-\ell+i} -q^{m+x-\ell} }{1-q} =[x]_{\ell-i} [x+m-\ell] - q^m[x]_{\ell-i+1}.
\end{multline*}

We have that
\begin{multline*}
    F(x;\del{\lambda}{c}{j}) = \prod_{i=1}^{\lambda'_j} [x+\lambda_i -1 -\ell+i ] \prod_{i=\lambda'_j+1}^\ell [x+\lambda_i -\ell+i] = F(x;\lambda) \prod_{i=1}^{\lambda'_j} \frac{ [x-1 +\lambda_i -\ell+i ]}{[x+\lambda_i -\ell+i ]}.
\end{multline*}

Using Lemma~\ref{lem:F_ratio} and that $\del{\lambda}{c}{j} = \lambda$ for $j>\lambda_1$,
\begin{multline*}
    \sum_{j=1}^m q^{m-j} F(x;\del{\lambda}{c}{j}) = F(x;\lambda)\left( [m-\lambda_1]  + q^{m-\lambda_1} \sum_{j} q^{\lambda_1 - j}\prod_{i=1}^{\lambda'_j} \frac{ [x-1 +\lambda_i -\ell+i ]}{[x+\lambda_i -\ell+i ]} \right)\\
    =  F(x;\lambda)\left( [m-\lambda_1] +q^{m-\lambda_1}[x-\ell+\lambda_1] -q^{m-\lambda_1}q^{\lambda_1}[x] \frac{F(x-1;\lambda)}{F(x;\lambda) }\right) \\
    = F(x;\lambda)[x-\ell+\lambda_1 + m-\lambda_1] -q ^m[x]F(x-1;\lambda),
\end{multline*}
which is what we wanted to show and completes the proof.
\end{proof}
 We also have the following relationship between $q$-rook numbers obtained from removing a single row from $\lambda$. 
\begin{lemma}\label{lem:qrook_linear_2}
For all fixed $k$,
\begin{align*}
    \sum_{i=1}^n q^{i-1 +\lambda_i} \qrook{k}{\del{\lambda}{r}{i}}
    = [n] \qrook{k}{\lambda} - [k]\qrook{k}{\lambda}.
\end{align*}
\end{lemma}
\begin{proof}
First of all, notice that we can replace $n$ with $\ell$ since for $i>\ell$, $\del{\lambda}{r}{i} = \lambda$, $\lambda_i=0$ and $[n] - \sum_{i=\ell+1}^n q^{i-1} = [\ell]$.
Multiplying by $[y]_{\ell - k}=[y] [y-1]_{(\ell-1)-k}$ on both sides and summing over all $k$, the identity is equivalent to
\begin{align*}
   [y] \sum_{i=1}^\ell q^{i-1 +\lambda_i} \sum_k  \qrook{k}{\del{\lambda}{r}{i}} [y-1]_{(\ell-1)-k}
    = [n] \sum_k \qrook{k}{\lambda}[y]_{\ell-k}  - \sum_k [k]\qrook{k}{\lambda}[y]_{\ell-k}.
\end{align*}

Thus, the identity is then equivalent to the generating function identity
\begin{multline*}
  [y] \sum_{i=1}^\ell q^{i-1 +\lambda_i} F(y-1,\del{\lambda}{r}{i})  \\
    = [\ell]F(y; \lambda) - DF(y;\lambda) = [\ell]F(y,\lambda) +q^{-y+ \ell}[y-\ell]F(y,\lambda) -q^{\ell-y}[y]F(y-1,\lambda).
\end{multline*}

We have that for $i\leq \ell$
\begin{multline*}
    F(y-1,\del{\lambda}{r}{i}) = \prod_{j=1}^{\ell-1}[y-1 + (\del{\lambda}{r}{i})_j - (\ell -1)+i] = \prod_{j=1}^{i-1} [y-1+\lambda_j -\ell +1 +j ] \prod_{j=i+1}^{\ell} [y-1+ \lambda_j - (\ell-1)+j-1 ] \\
    = \frac{1}{[y-1+\lambda_i-\ell+i]} \prod_{j=1}^{i-1} \frac{[y+\lambda_j-\ell]}{[y-1+\lambda_j - \ell]} F(y-1,\lambda),
\end{multline*}
and so we reach the following equivalent identity
\begin{multline*}
   [y] \sum_{i=1}^\ell  \frac{q^{y- \ell + i-1 +\lambda_i}}{[y-1+\lambda_i-\ell+i]} \prod_{j=1}^{i-1} \frac{[y+\lambda_j-\ell]}{[y-1+\lambda_j - \ell]} F(y-1,\lambda)  \\ = q^{y-\ell}[\ell]F(y,\lambda) +[y-\ell]F(y,\lambda) - [y]F(y-1,\lambda).
  \end{multline*}
We can rewrite it as
\begin{multline*}
  [y] \sum_{i=1}^\ell  \frac{[y- \ell + i  +\lambda_i] - [y-1 -\ell+i +\lambda_i]}{[y-1+\lambda_i-\ell+i]} \prod_{j=1}^{i-1} \frac{[y+\lambda_j-\ell +j]}{[y-1+\lambda_j - \ell+j]} F(y-1,\lambda)  \\
   = q^{y-\ell}[\ell]F(y,\lambda) +[y-\ell]F(y,\lambda) - [y]F(y-1,\lambda),
  \end{multline*}
 which reduces again to 
 \begin{align*}
     [y] F(y-1,\lambda)   \sum_{i=1}^\ell   \left( \prod_{j=1}^{i} \frac{[y+\lambda_j-\ell+j]}{[y-1+\lambda_j - \ell+j]} - \prod_{j=1}^{i-1} \frac{[y+\lambda_j-\ell+j]}{[y-1+\lambda_j - \ell+j]} \right) 
   = [y]F(y,\lambda) - [y]F(y-1,\lambda).
\end{align*}
After canceling the terms in the telescoping sum on the LHS, the identity reduces to
\begin{multline*}
   [y] F(y-1,\lambda)     \left( \prod_{j=1}^{\ell} \frac{[y+\lambda_j-\ell+j]}{[y-1+\lambda_j - \ell+j]} - 1\right) -   [y]F(y,\lambda) + [y]F(y-1,\lambda) \\
    =[y] (F(y,\lambda) -F(y-1,\lambda)) - [y]F(y,\lambda) + [y]F(y-1,\lambda)=0,
\end{multline*}
which completes the proof.
\end{proof}

\begin{lemma}\label{lem:qrook_identity1}
We have the following identity for $q$-rook numbers:
$$\sum_{(i,j) \in \lambda} q^{i-j+\lambda_i} \qrook{k}{\lambda/(i,j)} = q[k+1] \qrook{k+1}{\lambda}.$$
\end{lemma}

\begin{proof}
Translating this identity into generating functions, we have that it is equivalent to 
\begin{align*}
    \sum_{k=0}^{\ell-1} \sum_{(i,j)\in\lambda} q^{i-j+\lambda_i}\qrook{k}{\lambda/(i,j)} [x]_{\ell-1-k} = \sum_{k=0}^{\ell-1} q[k+1]\qrook{k+1}{\lambda}[x]_{\ell-k-1}= \sum_{t=0}^{\ell} q[t]\qrook{t}{\lambda}[x]_{\ell-t},
\end{align*}
which can be rewritten, with the help of Lemma~\ref{lem:DF}, in terms of the $F$ function as
\begin{align}\label{eq:quadratic_rook_f_identity}
    \sum_{(i,j)\in\lambda} q^{i-j+\lambda_i} F(x;\lambda/(i,j)) = q DF(x;\lambda) = -q^{\ell-x+1} [x-\ell]F(x;\lambda) + q^{1+\ell-x}[x] F(x-1;\lambda).
\end{align}
Now, we notice that if $\mu = \lambda/(i,j)$
\begin{align*}
    F(x;\lambda/(i,j)) &= \prod_{t=1}^{\ell-1} [x+\mu_t - \ell+1+t] \\
    &= \prod_{t=1}^{i-1}[x+\lambda_t -1 -\ell+1+t] \prod_{t=i}^{\lambda_j'-1}[x+\lambda_{t+1}-1-\ell+1+t] \prod_{t=\lambda_j'}^{\ell-1}[x+\lambda_{t+1}-\ell+1+t]\\
    &= F(x;\lambda) \frac{ \prod_{t=i+1}^{\lambda_j'}[x+\lambda_{t}-1-\ell+t] }{\prod_{t=i}^{\lambda_j'}[x+\lambda_{t}-\ell+t]}.
\end{align*}

Note that $q^{x-1-\ell +i+\lambda_i}=[x+\lambda_i-\ell+i] - [x+\lambda_i-\ell+i-1]$, so we can rewrite this as
\begin{multline*}
    q^{i-j+\lambda_i} F(x;\lambda/(i,j)) \\= q^{-x+1+\ell-j} F(x;\lambda) ([x+\lambda_i-\ell+i] - [x+\lambda_i-\ell+i-1]) \frac{ \prod_{t=i+1}^{\lambda_j'}[x+\lambda_{t}-1-\ell+t] }{\prod_{t=i}^{\lambda_j'}[x+\lambda_{t}-\ell+t]}\\
    = q^{-x+1+\ell-j} F(x;\lambda) \left( \prod_{t=i+1}^{\lambda_j'}\frac{[x+\lambda_{t}-1-\ell+t] }{[x+\lambda_{t}-\ell+t]} - \prod_{t=i}^{\lambda_j'} \frac{[x+\lambda_{t}-1-\ell+t] }{[x+\lambda_{t}-\ell+t]} \right).
\end{multline*}
Fixing $j$ and summing over all possible $i=1\ldots \lambda_j'$ we  get telescoping cancellations and so
\begin{align*}
    \sum_{(i,j)\in\lambda} q^{i-j+\lambda_i} F(x;\lambda/(i,j)) = \sum_{j} q^{-x+1+\ell-j} F(x;\lambda)\left( 1 - \prod_{t=1}^{\lambda_j'} \frac{[x+\lambda_{t}-1-\ell+t] }{[x+\lambda_{t}-\ell+t]} \right). 
    \end{align*}
Substituting this into the LHS of~\eqref{eq:quadratic_rook_f_identity}, the needed identity transforms to the equivalent
\begin{align*}
    \sum_{j} q^{1-j+\ell-x} F(x;\lambda)\left(1-\prod_{t=1}^{\lambda_j'} \frac{[x+\lambda_{t}-1-\ell+t] }{[x+\lambda_{t}-\ell+t]} \right) &= q DF(x;\lambda) \\ &= q^{\ell-x+1}\left( [x]F(x-1;\lambda) - [x-\ell]F(x;\lambda)\right), 
\end{align*}
which is equivalent to
  \[     [x] \frac{ F(x-1;\lambda)}{F(x;\lambda)} =  [x-\ell] +  \sum_{j=1}^{\lambda_1} q^{-j}  -  q^{-j} \prod_{t=1}^{\lambda_j'} \frac{[x+\lambda_{t}-1-\ell+t] }{[x+\lambda_{t}-\ell+t]}. \]
     This last identity simplifies as
\[       q^{\lambda_1} [x] \frac{ F(x-1;\lambda)}{F(x;\lambda)} =  [x-\ell+\lambda_1]  -  \sum_{j=1}^{\lambda_1} q^{\lambda_1-j}\prod_{t=1}^{\lambda_j'} \frac{[x+\lambda_{t}-1-\ell+t] }{[x+\lambda_{t}-\ell+t]}, \]
    and is equivalent to the formula in Lemma~\ref{lem:F_ratio}, which completes the proof. \end{proof}

  \subsection{$q$-hit identities}
  
  We now translate the above rook identities into $q$-hit identities using the relationship from Definition~\ref{lem: hit in terms of rs}. Let 
$$G^{m,n}(x;\lambda)  =  \sum_{k=0}^n \qhit{k}{m,n}{\lambda} (q^{x})^k.$$
Then we have the following equivalences in terms of the generating functions
\begin{align}\label{eq:qhit_gen_fun_rooks}
    G^{m,n}(x;\lambda) 
    &:= \sum_{i=0}^n \sum_{k=0}^i \dfrac{q^{\binom{k}{2}-|\lambda|}}{\qfactorial{m-n}} \qrook{i}{\lambda} \qfactorial{m-i} \qbinom{i}{k} (-1)^{i+k} q^{mi-\binom{i}{2}}q^{xk}\\
    &= \frac{q^{-|\lambda|}}{\qfactorial{m-n} }  \sum_{i=0}^n\qrook{i}{\lambda} \qfactorial{m-i} q^{mi-\binom{i}{2}} (-1)^i \sum_{k=0}^i q^{\binom{k}{2}} \qbinom{i}{k} (-q^x)^k \nonumber \\
    &= \frac{q^{-|\lambda|}}{\qfactorial{m-n} }  \sum_{i=0}^n \qrook{i}{\lambda} \qfactorial{m-i} q^{mi-\binom{i}{2}} (-1)^i  \prod_{k=0}^{i-1}(1-q^{x+k}). \nonumber
\end{align}

\begin{lemma}\label{lem:qhit_linear_1}
For every $\lambda$ inside an $n\times m$ board and for every $k$
 $$\sum_{j=1}^{m}  q^{m+n-j - \lambda_j'} \qhit{k}{m-1,n}{\del{\lambda}{c}{j}} = [m-n] \qhit{k}{m,n}{\lambda} q^{n-k}.$$
\end{lemma}

\begin{proof}
The collection of identities for $k=0,\ldots,n$ is equivalent to the following generating function identities 
\begin{multline*}
    \sum_{k=0}^n \sum_{j=1}^{m}  q^{m+n-j - \lambda_j'} \qhit{k}{m-1,n}{\del{\lambda}{c}{j}}q^{xk} \\ = [m-n] \sum_{k=0}^n \qhit{k}{m,n}{\lambda} q^{n-k}q^{xk}
    \sum_{j=1}^m q^{m+n-j-\lambda_j'} G^{m-1,n}(x;\del{\lambda}{c}{j}) = q^n[m-n] G^{m,n}(x-1;\lambda).
\end{multline*}

Using~\eqref{eq:qhit_gen_fun_rooks} and expanding in the $q^x$-polynomial basis $(q^x;q)_k=\prod_{k=0}^{i-1}(1-q^{x+k})$ for $i=0,\ldots,n$, the $G$-identity is equivalent to the following rook identity for every $i$:
\begin{multline*}
\sum_{j=1}^{m}  q^{m+n-j - \lambda_j'} \frac{q^{-|\lambda|+\lambda_j'}}{\qfactorial{m-1-n} }  \sum_{i=0}^n \qrook{i}{\del{\lambda}{c}{j}} \qfactorial{m-1-i} q^{mi-i-\binom{i}{2}} (-1)^i  \prod_{k=0}^{i-1}(1-q^{x+k})  \\
 =q^n[m-n] \frac{q^{-|\lambda|}}{\qfactorial{m-n} }  \sum_{i=0}^n \qrook{i}{\lambda} \qfactorial{m-i} q^{mi-\binom{i}{2}} (-1)^i  \prod_{k=0}^{i-1}(1-q^{x-1+k}),
 \end{multline*}
 which simplifies as
\begin{multline*}  \sum_{j=1}^{m}  q^{m-j }  \sum_{i=0}^n \qrook{i}{\del{\lambda}{c}{j}} \qfactorial{m-1-i} q^{mi-i-\binom{i}{2}} (-1)^i  \prod_{k=0}^{i-1}(1-q^{x+k})  \\
 = \sum_{i=0}^n \qrook{i}{\lambda} \qfactorial{m-i} q^{mi-\binom{i}{2}-i} (-1)^i  q^i\prod_{k=0}^{i-1}(1-q^{x-1+k}).
\end{multline*}
Notice that 
\[
q^i(q^{x-1};q)_i = (q^i-q^{x+i-1})(q^x;q)_{i-1}=(q^i-1)(q^x;q)_{i-1} +(q^x;q)_i,
\]
and so the RHS above expands in the $(q^x;q)_i$ basis as:
\begin{multline*}
   \sum_{i=0}^n \qrook{i}{\lambda} \qfactorial{m-i} q^{mi-\binom{i+1}{2}} (-1)^i ( (q^i-1)(q^x;q)_{i-1} +(q^x;q)_i )\\
= \sum_{i=0}^n (-1)^i (q^x;q)_i \left(    \qrook{i}{\lambda} \qfactorial{m-i} q^{mi-\binom{i+1}{2}}   -  \qrook{i+1}{\lambda} \qfactorial{m-i-1} q^{mi+m-\binom{i+2}{2}} (q^{i+1}-1)\right).
\end{multline*}

Therefore, this is equivalent to the following $q$-rook identity for every $i$:
\begin{multline*}
      \sum_{j=1}^{m}  q^{m-j }   \qrook{i}{\del{\lambda}{c}{j}} \qfactorial{m-1-i} q^{mi-\binom{i+1}{2}} \\ = \left(    \qrook{i}{\lambda} \qfactorial{m-i} q^{mi-\binom{i+1}{2}}   -  \qrook{i+1}{\lambda} \qfactorial{m-i-1} q^{mi+m-\binom{i+2}{2}}  (q^{i+1}-1)\right).
\end{multline*}
Simplifying last expression, we obtain
\begin{align*}
      \sum_{j=1}^{m}  q^{m-j }   \qrook{i}{\del{\lambda}{c}{j}}  =    \qrook{i}{\lambda} [m-i]    -  \qrook{i+1}{\lambda}  (q^m -q^{m-i-1}),
\end{align*}
which is exactly Lemma~\ref{lem:qrook_linear_1}.
\end{proof}

\begin{lemma}\label{lem:qhit_linear_2}
Let $k\leq n\leq m$ be fixed, and $\lambda \subset n\times m$. We have the following $q$-hit identity: 
\[
[m-n+1]\sum_{i=1}^{n}  q^{i-1} \qhit{k}{m,n-1}{\del{\lambda}{r}{i}} =  \qhit{k}{m,n}{\lambda}q^k[n-k] + \qhit{k+1}{m,n}{\lambda}[k+1].
\]
\end{lemma}

\begin{proof}
Multiplying both sides by $(q^x)^k$ and summing over all $k$, the identity becomes equivalent to the following generating function identity:
\begin{align*}%\label{gs second linear relation}
 [m-n+1]\sum_{i=1}^{n}  q^{i-1} G^{m,n-1}(x;\del{\lambda}{r}{i}) = \sum_{k=0}^{n-1} \qhit{k}{m,n}{\lambda} (q^{x+1})^k [n-k] + \sum_{k=0}^{n-1} \qhit{k+1}{m,n}{\lambda}[k+1]q^{xk}.   
\end{align*}
To prove the above identity, we will use the following difference operator and its properties:
$$\Delta F(x): = \frac{F(x+1) - F(x)}{q^x(q-1)}, \quad \Delta (q^x)^p = [p] (q^x)^{p-1}, \quad \Delta(q^x,q)_p = -[p] (q^{x+1},q)_{p-1} = -\frac{ [p] (q^x;q)_p}{1-x}. $$
Notice also that $q^k[n-k] = [n]-[k]$ and $[k+1]q^{xk} = \Delta q^{x (k+1)}$. Therefore, we have
\begin{align*}
    \sum_{k=0}^{n-1} \qhit{k}{m,n}{\lambda} \left(q^{x+1}\right)^k [n-k]  &= \sum_{k=0}^{n-1} \qhit{k}{m,n}{\lambda} q^{xk} [n] - \sum_{k=0}^{n-1} \qhit{k}{m,n}{\lambda} [k] q^{xk}  
 \\ &= [n] G^{m,n}(x; \lambda)  - q^x \Delta G^{m,n}(x;\lambda)
\intertext{and}
    \sum_{k=0}^{n-1} \qhit{k+1}{m,n}{\lambda}[k+1]q^{xk}  &= \Delta G^{m,n}(x;\lambda).
\end{align*}
Thus, the generating function identity is equivalent to
\begin{align*}
    [m-n+1] \sum_{i=1}^n q^{i-1} G^{m,n-1}(x; \del{\lambda}{r}{i}) = [n] G^{m,n}(x; \lambda)  - q^x \Delta G^{m,n}(x;\lambda) + \Delta G^{m,n}(x;\lambda). 
\end{align*}
Using the formula 
\[
(1-q^x) \Delta G^{m,n}(x;\lambda) = \frac{q^{-|\lambda|}}{[m-n]!}\sum_{k=0}^n [k]\qrook{k}{\lambda} [m-k]! q^{mk-\binom{k}{2}} (-1)^{k-1}\underbrace{(1-q^x)(q^{x+1},q)_{k-1}}_{(q^x;q)_k},
\]
we can express the generating function identity in terms of $q$-rook numbers generating function as:
\begin{multline*}
    [m-n+1] \sum_{i=1}^n q^{i-1}  \frac{q^{-|\lambda| +\lambda_i}}{[m-n+1]!}\sum_{k=0}^n \qrook{k}{\del{\lambda}{r}{i}}[m-k]! q^{mk-\binom{k}{2}} (-1)^{k}(q^x,q)_{k} \\
    = [n] \frac{q^{-|\lambda|}}{[m-n]!} \sum_k \qrook{k}{\lambda} [m-k]! q^{mk-\binom{k}{2}} (-1)^k (q^x,q)_k \\ +  
    \frac{q^{-|\lambda|}}{[m-n]!}\sum_{k=0}^n [k]\qrook{k}{\lambda} [m-k]! q^{mk-\binom{k}{2}} (-1)^{k-1}(q^x;q)_k.
\end{multline*}
For each $k$,  the coefficients at $(q^x;q)_k$ coincide, after canceling common factors, and reducing then to the $q$-rook identities from Lemma~\ref{lem:qrook_linear_2}
\begin{align*}
    \sum_{i=1}^n q^{i-1 +\lambda_i} \qrook{k}{\del{\lambda}{r}{i}}
    = [n] \qrook{k}{\lambda} - [k]\qrook{k}{\lambda}.
\end{align*}
\end{proof}

\begin{lemma}\label{lem:qhit_quadratic}
We have the  following identity
\[
 q^k\sum_{(i,j)\in \lambda} q^{i+(m-j-\lambda'_j)} \qhit{k}{m-1,n-1}{\lambda/(i,j)} = [k+1] \qhit{k+1}{m,n}{\lambda}.
\]

\end{lemma}

\begin{proof}
We show that this result follows from Lemma~\ref{lem:qrook_identity1}.
By~\eqref{eq: hit in terms of rs},
\begin{align*}
    [k+1] \qhit{k+1}{m,n}{\lambda} &= [k+1] \frac{q^{\binom{k+1}{2} - |\lambda|}}{[m-n]!} \sum_{t=k+1}^n \qrook{t}{\lambda} [m-t]! \qbinom{t}{k+1} (-1)^{t+k+1} q^{mt-\binom{t}{2}} \\
    &=  \frac{q^{\binom{k+1}{2} - |\lambda|}}{[m-n]!} \sum_{t'=k}^{n-1} [t'+1]\qrook{t'+1}{\lambda} [m-t'-1]! \qbinom{t'}{k} (-1)^{t'+k} q^{mt'+m-\binom{t'+1}{2}},
\end{align*}
where we reindexed the sum with $t'=t-1$. Next, we apply Lemma~\ref{lem:qrook_identity1} and exchange sums to obtain
\begin{multline*}
   [k+1] \qhit{k+1}{m,n}{\lambda} = \frac{q^{\binom{k+1}{2} - |\lambda|}}{[m-n]!} \sum_{(i,j)\in \lambda} q^{m-j+\lambda_i} \sum_{t'=k}^{n-1} \qrook{t'}{\lambda/(i,j)} [m-t'-1]! \qbinom{t'}{k} (-1)^{t'+k} q^{mt'+m-\binom{t'+1}{2}}\\
   =q^k \sum_{(i,j) \in \lambda} q^{i+m-j-\lambda'_j} \frac{q^{\binom{k}{2}}-|\lambda/(i,j)|}{[m-n]!} \sum_{t'=k}^{n-1} \qrook{t'}{\lambda/(i,j)} [m-t'-1]! \qbinom{t'}{k} (-1)^{t'+k} q^{(m-1)t'-\binom{t'}{2}}\\
   =q^k \sum_{(i,j) \in \lambda} q^{i+m-j-\lambda'_j} \qhit{k}{m-1,n-1}{\lambda/(i,j)},
\end{multline*}
where we also used that $|\lambda/(i,j)| = |\lambda|-\lambda_i-\lambda'_j+1$ and~\eqref{eq: hit in terms of rs} for $\qhit{k}{m-1,n-1}{\lambda/(i,j)}$.
\end{proof}

\section{The Guay-Paquet $q$-hit identity} \label{sec: pf of MGP}

In this section we give our main result, a proof of Theorem~\ref{thm:qhitCSFabelian} using the $q$-rook theory identities from Section~\ref{sec:new_qrooks}. We start by giving an example of this elegant identity.

\begin{example} \label{ex:tcase}
For $\lambda = (2,1)$ inside a $2\times 3$ board, looking at Figure~\ref{fig:hitexB}, we see that
$\qhit{0}{3,2}{\lambda}=q^0=1$, $\qhit{1}{3,2}{\lambda}=2q+2q^2$,
$\qhit{2}{3,2}{\lambda}=q^3$. One can verify that  
\[
\csft{21}{q} = \dfrac{1}{[3][2]}  \left(
  \csft{3^0}{q} + (2q^2+2q)\csft{3^1}{q}+q^3\csft{3^2}{q}\right). 
\]
\end{example}

\begin{remark}
By Lemma~\ref{lem:q-hit for m,m}, the identity in Theorem~\ref{thm:qhitCSFabelian} can be rewritten as 
\begin{equation} \label{eq: MGP relation square boards}
\csft{\lambda}{q} \,=\, \dfrac{1}{\qfactorial{m}} \sum_{j=0}^{n}
\qhit{j}{m}{\lambda} \cdot \csft{m^j}{q}.
\end{equation}
\end{remark}

\begin{figure}
     \centering
     \begin{subfigure}[b]{0.23\textwidth}
     \raisebox{25pt}{\includegraphics{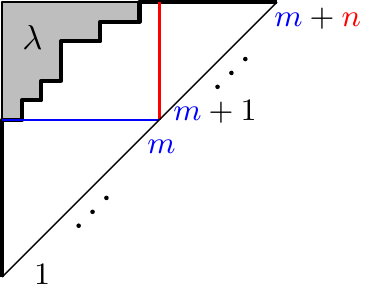}}
     \caption{}
    \label{fig:hitexA}
     \end{subfigure}
     \begin{subfigure}[b]{0.7\textwidth}
         \centering
         \includegraphics{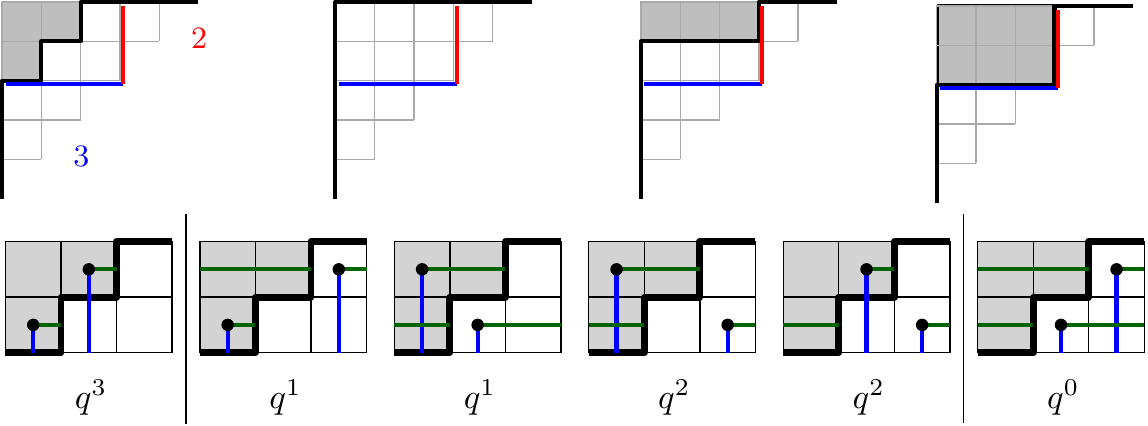}
         \caption{}
         \label{fig:hitexB}
     \end{subfigure}
    \caption{(A) an abelian Dyck path $\lambda$ inside an $n\times m$ board. (B) Top: the paths for $\lambda=(2,1)$, and for the rectangles $3^0,3^1,3^2$ inside a $2\times 3$ board. (B) Bottom: the six placements of $2$ rooks in $2\times 3$
  divided by how many rooks ``hit''  $(2,1)$ (in gray) and the
  associated statistic  to each rook placement.} 
\label{fig:hitex}
\end{figure}

Our main self-contained proof is an inductive argument, where the induction is applied both on the size $m+n$ of the graph, and also on the number of variables in the symmetric polynomials.
Namely, we consider the chromatic symmetric functions in variables $x_1,\ldots,x_M$ and each monomial appearing as a particular assignment of the variables (i.e. colors) to the vertices. That is, the vertices $1,\ldots,N=m+n$ are colored $\{1,\ldots,M\}$. For simplicity, we denote by $X_\lambda^N(M)$ the chromatic symmetric polynomial $X_{G(\lambda)}(x_1,\ldots,x_M;q)$ where the graph $G(\lambda)$ has  $N$ vertices. We will use induction on both $M$ and $n,\ m$ when necessary, driven by the following recursion.

Recall that $\lambda/(i,j)$ is the partition obtained by removing the $i$th row and the $j$th column from $\lambda$. Moreover, we denote by $\lambda/i$ the partitions obtained by removing from $\lambda$ the $i$th column, for $i=1,\ldots,m$, or the $(m+n-i+1)$th row, for $i=m+1,\ldots,m+n$.

\begin{lemma}\label{lem:X_recursion}
For $\lambda \subset n \times m$ we have the following recursion
\begin{align*}
    X_{\lambda}^{m+n}(M) =& X_{\lambda}^{m+n}(M-1) 
    + x_M\sum_{i=1}^{m+n} q^{m+n-i - \lambda_i'} X^{m+n-1}_{\lambda/i}(M-1) \\
    &+ x_M^2\sum_{(i,j) \in \lambda} q^{ i-1 + (m+n-j-\lambda_j')} X^{m+n-2}_{\lambda/(i,j)}(M-1).
\end{align*}
\end{lemma}

\begin{proof}
In the abelian case, i.e. when $\lambda \subset n\times m$, the graph $G(\lambda)$ consists of a clique with vertices $\{1,\ldots, m\}$, a clique with vertices $\{m+1,\ldots,m+n\}$ and a bipartite graph in between with edges $(i,m+j)$ for each $(i,j)$ in $\{1,\ldots,n\}\times \{1,\ldots,m\} \setminus \lambda$ (the complement of $\lambda$ in $n\times m$). Therefore, a coloring of this graph has at most two vertices of the same color. If the colors used are in $\{1,\ldots,M\}$, there are three cases for the appearances of color $M$:
\begin{compactitem}
\item[\textbf{1.}] No vertex is colored $M$, this term contributes  $X_\lambda^{m+n-1}(M-1)$ to $X_\lambda^{m+n}(M)$.
\item[\textbf{2.}] Only one vertex is colored $M$. Suppose this vertex is in column $j$ (from left) and row $i=N-j$ (from top to bottom). It creates ascents with all vertices above it  but not in $\lambda$, giving $N-j-\lambda_j'$ ascents. Deleting this vertex corresponds to deleting its row and column (only one would be a row/column of $\lambda$) and we get a graph on $N-1$ vertices with shape $\lambda/j$ (deleting either row $N-j$ from $\lambda$, or  column $j$ from $\lambda$). The remaining vertices and their ascents are not affected by this, so all their possible colorings contribute 
    $$x_M \sum_{j=1}^{N} q^{N-j-\lambda'_j}  X_{\lambda/j}^{m+n-1}(M-1).$$
    
\item[\textbf{3.}] Two vertices are colored $M$. Suppose that the lower one is in column $j$ and the higher  one is in row $i$ (counting from the top), necessarily with $(i,j) \in \lambda$. The lower vertex contributes $N-j - \lambda_j'$ ascents with the ``visible'' vertices above it. The higher vertex contributes $i-1$ ascents, the number of vertices above it, giving a total of $N-j - \lambda_j'+ i-1$ ascents. We can  remove these two vertices, by removing row $i$ and column $j$ from $\lambda$ and decreasing $N$ by 2. Again, the remaining vertices and their ascents are not affected, so
these terms contribute
$$x_M^2\sum_{(i,j)} q^{N-j-\lambda'_j + i-1 }  X^{m+n-2}_{\lambda/(i,j)}(M-1).$$
\end{compactitem}
\end{proof}
For rectangular shapes $\lambda = (m^k)$, Lemma~\ref{lem:X_recursion} simplifies to give  the following recursive expansion.
\begin{lemma}\label{lem:X_rect_recursion}
\begin{align*}
 X^{m+n}_{m^k}(M) &= X^{m+n}_{m^k}(M-1) + x_M \left(q^{n-k}[m] X^{m+n-1}_{(m-1)^k}(M-1) + [k]X^{m+n-1}_{m^{k-1}}(M-1) \right. \\  &\left. + q^k[n-k] X^{m+n-1}_{m^k}(M-1) \right)+
x_M^2 q^{n-k}[k][m] X^{m+n-2}_{(m-1)^{k-1}}(M-1). 
\end{align*}
\end{lemma}

\begin{proof}
This follows by carefully applying Lemma~\ref{lem:X_recursion} to the shape $\lambda = m^k$. If $i\in[1,m]$ then $\lambda/i = (m-1)^k$ and $N-i - \lambda_i' = m+n-i -k$. If $i\in [m+1,m+n-k]$, then it is not a row or column of $\lambda$ and we have $\lambda/i = m^k$ with $m+n-i-0$ ascents.  If $i \in[m+n-k+1,m+n]$ then $\lambda/i = m^{k-1}$ and there are $m+n-i$ ascents. For $(i,j) \in \lambda$ we always have $\lambda/(i,j) = (m-1)^{k-1}$ and the ascents are $m+n-j-k + i-1$. Summing over the row/column indices in the given intervals we get the desired ascent statistics as the given $q$-integers.
\end{proof}

\begin{proof}[Proof of Theorem~\ref{thm:qhitCSFabelian}]
Translating Theorem~\ref{thm:qhitCSFabelian} into chromatic symmetric polynomials, we want to prove that for every $M$ we have
\begin{align}\label{eq:GPpolynomials}
X_\lambda^{m+n}(M) =\dfrac{1}{\qfalling{m}{n}} \sum_{j=0}^{n}
\qhit{j}{m,n}{\lambda} \cdot X_{m^j}^{m+n}(M).
\end{align}
If $M < m$ then both $X_\lambda^{m+n}(M)=0$ and $X_{m^j}^{m+n}(M)=0$ since there is no proper coloring of the lower complete graph on $m$ vertices. Otherwise, if $M=1$  then $m=1$ and so $n=1$ and if $\lambda= \emptyset$ then  $X_{\lambda}^{2}(M) \neq 0$ with $X_{m^0}(1) = 0$ and $\qhit{1}{1,1}{\emptyset} =0$, and the identity is satisfied. If $\lambda=(1)$ then $X^{2}_{(1)}(1) = x_1^2$, $[m]_n=[1]_1=1$, $\qhit{0}{1,1}{(1)}=0$ and $\qhit{1}{1,1}{(1)} = 1$ with $m^1=(1)$, and the identity is again trivially satisfied. This completes the initial conditions for the recursion in Lemma~\ref{lem:X_recursion}.

We will prove identity~\eqref{eq:GPpolynomials} by induction on $M$. As the argument above shows, the identity is trivially satisfied for $M=1$. Suppose that~\eqref{eq:GPpolynomials} is true for $M-1$, every $m\geq n$, and every shape $\lambda \subset m^n$. Naturally, if $M<m$ then both sides become trivially 0.

The rest of the proof is as follows. We apply Lemma~\ref{lem:X_rect_recursion} to each term $X_{m^j}(M)$ appearing in the RHS of~\eqref{eq:GPpolynomials}. We also apply   Lemma~\ref{lem:X_recursion} to the LHS of~\eqref{eq:GPpolynomials}, and the inductive hypothesis  to  $X_{\mu}^{\star}(M-1)$ for all appearing terms, where $\star$ means any value $\leq m+n$.

Applying Lemma~\ref{lem:X_rect_recursion} to each term $X_{m^j}(M)$ appearing in the RHS in~\eqref{eq:GPpolynomials}, we obtain
\begin{multline*}
    \dfrac{1}{\qfalling{m}{n}} \sum_{j=0}^{n}
\qhit{j}{m,n}{\lambda} \cdot X_{m^j}^{m+n}(M) =  \frac{1}{[m]_n} \sum_k \qhit{k}{m,n}{\lambda} \left( X^{m+n}_{m^k}(M-1) + x_M \left(q^{n-k}[m] X^{m+n-1}_{(m-1)^k}(M-1) \right.\right. \\  \left. \left. +[k]X^{m+n-1}_{m^{k-1}}(M-1) + q^k[n-k] X^{m+n-1}_{m^k}(M-1) \right)+
x_M^2 q^{n-k}[k][m] X^{m+n-2}_{(m-1)^{k-1}}(M-1) \right).
\end{multline*}

We now apply Lemma~\ref{lem:X_recursion} to the LHS in~\eqref{eq:GPpolynomials}. We split the sum in the linear $x_M$ term into $j \in [1,m]$, when $\lambda/j = \del{\lambda}{c}{j}\subset (m-1)\times n$ (removing column $j$), and then $i=m+n+1-j \in [1,n]$, where $\lambda_j'=0$ and $m+n-j - \lambda_j'=i-1$ and $\lambda/i \subset (n-1) \times m$. We then apply the inductive hypothesis to each $X_\mu^\star(M-1)$ appearing with the corresponding rectangular frame, obtaining 
\begin{multline*}
     X_{\lambda}^{m+n}(M)  
     =X_{\lambda}^{m+n}(M-1) \\
    + x_M\sum_{j=1}^{m+n} q^{m+n-j - \lambda_j'} X^{m+n-1}_{\lambda/j}(M-1) 
    + x_M^2\sum_{(i,j) \in \lambda} q^{ i-1 + (m+n-j-\lambda_j')} X^{m+n-2}_{\lambda/(i,j)}(M-1) \\
\end{multline*}
\begin{multline*}
    =   \frac{1}{[m]_n} \sum_k \qhit{k}{m,n}{\lambda} X^{m+n}_{m^k}(M-1) 
    + x_M\sum_k \frac{ \sum_{j=1}^{m} q^{m+n-j - \lambda_j'}  \qhit{k}{m-1,n}{\del{\lambda}{c}{j}}  }{[m-1]_n}  X^{m+n-1}_{(m-1)^k}(M-1) \\
    + x_M \sum_k \frac{ \sum_{i=1}^n q^{i-1}  \qhit{k}{m,n-1}{\lambda/i} }{[m]_{n-1}} X^{m+n-1}_{m^k}(M-1)\\ 
    + x_M^2 \sum_k \frac{ \sum_{(i,j) \in \lambda} q^{ i-1 + (m+n-j-\lambda_j')}  \qhit{k}{m-1,n-1}{\lambda/(i,j)} }{[m-1]_{n-1}} X^{m+n-2}_{(m-1)^k}(M-1).
\end{multline*}

Applying Lemmas~\ref{lem:qhit_linear_1},~\ref{lem:qhit_linear_2} and~\ref{lem:qhit_quadratic} to the sums of $q$-hit numbers above, we get that
\begin{align*}
     X_{\lambda}^{m+n}(M) %= \\
    &=   \frac{1}{[m]_n} \sum_k \qhit{k}{m,n}{\lambda} X^{m+n}_{m^k}(M-1) 
    + x_M\sum_k \frac{ q^{n-k} [m-n] \qhit{k}{m,n}{\lambda}  }{[m-1]_n}  X^{m+n-1}_{(m-1)^k}(M-1) \\
    &+ x_M \sum_k \frac{ q^k[n-k] \qhit{k}{m,n}{\lambda} + [k+1] \qhit{k+1}{m,n}{\lambda} }{[m-n+1] [m]_{n-1}} X^{m+n-1}_{m^k}(M-1) \\
     &+ x_M^2 \sum_k \frac{ q^{n-1-k}[k+1] \qhit{k+1}{m,n}{\lambda} }{[m-1]_{n-1}} X^{m+n-2}_{(m-1)^k}(M-1).
\end{align*}
Simplifying the factors $\frac{ [m-1]_n}{[m-n]} = \frac{ [m]_n}{[m]}$ , $[m-n+1][m]_{n-1} = [m]_n$ and $[m-1]_{n-1} = \frac{[m]_n}{[m]}$, and grouping the terms with $\qhit{k}{m,n}{\lambda}$, we can rewrite the above identity as
\begin{multline*}
     X_{\lambda}^{m+n}(M) 
    =   \frac{1}{[m]_n} \sum_k \qhit{k}{m,n}{\lambda} \times \left( X^{m+n}_{m^k}(M-1) 
    + x_M  q^{n-k} [m]  X^{m+n-1}_{(m-1)^k}(M-1) \right. \\
\left.    + x_M ( q^k[n-k] X^{m+n-1}_{m^k}(M-1)  + [k] X^{m+n-1}_{m^{k-1}}(M-1) )
    + x_M^2  q^{n-k}[k] [m]  X^{m+n-2}_{(m-1)^{k-1}}(M-1) \right)\\
    = \frac{1}{[m]_n} \sum_k \qhit{k}{m,n}{\lambda} X_{m^k}^{m+n}(M),
\end{multline*}
where we recognized the sum in the parentheses as the RHS of the recursion for rectangular CSF, namely Lemma~\ref{lem:X_rect_recursion}. This completes the induction. 

\end{proof}

\section{The Abreu--Nigro expansion in the elementary basis} \label{sec: MGP to AN}

In this section we show that Guay-Paquet's identity (Theorem~\ref{thm:qhitCSFabelian}) is equivalent to Abreu--Nigro's identity (Theorem~\ref{AN:generalLambda}). 
We start by giving a proof of Abreu--Nigro's identity for rectangular shapes. 
\begin{lemma}[Abreu--Nigro's formula for rectangles]\label{lem:ANrectangle}
\begin{align*}
    \csft{m^k}{q} &=\qfactorial{k}\qhit{k}{m+n-k}{m^k}\cdot e_{m+n-k,k} 
    + \sum_{r=0}^{k-1}q^r \qfactorial{r}\qnumber{m+n-2r} \qhit{r}{m+n-r-1}{m^k} \cdot e_{m+n-r,r}.
\end{align*}
\end{lemma}

In order to prove this case of the Abreu--Nigro identity we need the following result.

\begin{lemma}[Guay-Paquet formula for rectangles] 
For the shape $(m-1)^k \subset n\times m$ we have that,
\label{lem: MGP for rectangles}
\begin{equation*} %\label{eq: MGP for rectangles}
[m]X_{(m-1)^k} =q^k[m-k]X_{m^k} + [k]X_{m^{k-1}}.
\end{equation*}
\end{lemma}

\begin{corollary}
$$
\qbinom{m-1}{k} X_{m^k} = \sum_{j=0}^k \qbinom{m}{j} (-1)^{k-j} q^{-(k+j)(k-j+1)/2} X_{(m-1)^j}.
$$
\end{corollary}

\begin{proof}[Proof of Lemma~\ref{lem: MGP for rectangles}]
By Theorem~\ref{thm:qhitCSFabelian} for the shape $\lambda=(m-1)^k \subset n\times m$ and the formula for the $q$-hit numbers $\qhit{r}{m,n}{(m-1)^k}$ from Proposition~\ref{prop: hits small rect in rect} we obtain
\begin{align*}
X_{(m-1)^k} &= \frac{1}{[m]_{n}} q^k[m-1]_k [m-k]_{n-k} X_{m^k} + \frac{1}{[m]_{n}}[k][m-1]_{k-1}[m-k]_{n-k} X_{m^{k-1}},\\
\intertext{which simplifies to }
[m]X_{(m-1)^k} &= q^k [m-k] X_{m^k} + [k]X_{m^{k-1}}.
\end{align*}
\end{proof}

\begin{proof}[Proof of Lemma~\ref{lem:ANrectangle}]
We use induction on $m$ and $k$. For the base case, note that $X_{m^0} = [m+n]! e_{m+n} = \qhit{0}{m+n}{m^0} e_{m+n}$. Now, by Lemma~\ref{lem: MGP for rectangles} we have that 
\[
q^k[m-k]X_{m^k} = \left([m] X_{(m-1)^k} - [k]X_{m^{k-1}}\right).
\]
Next, we use the induction hypothesis on $X_{(m-1)^k}$ and $X_{m^{k-1}}$ and we simplify our expression, so we want to prove that 

\begin{multline}\label{eq:induction together}
q^k[m-k]\Big( \qfactorial{k}\qhit{k}{m+n-k}{m^k}\cdot e_{m+n-k,k}
    + \sum_{r=0}^{k-1}q^r \qfactorial{r}\qnumber{m+n-2r} \qhit{r}{m+n-r-1}{m^k} \cdot e_{m+n-r,r} \Big)\\ 
    = \qnumber{m}\Big( \qfactorial{k}\qhit{k}{m+n-k}{(m-1)^k}\cdot e_{m+n-k,k} 
    + \sum_{r=0}^{k-1}q^r \qfactorial{r}\qnumber{m+n-2r} \qhit{r}{m+n-r-1}{(m-1)^k} \cdot e_{m+n-r,r} \Big)\\
    - \qnumber{k}\Big(\qfactorial{k-1}\qhit{k-1}{m+n-k+1}{m^{k-1}}\cdot e_{m+n-k+1,k-1} 
    - \sum_{r=0}^{k-2}q^r \qfactorial{r}\qnumber{m+n-2r} \qhit{r}{m+n-r-1}{m^{k-1}} \cdot e_{m+n-r,r}\Big).
\end{multline}

We do so by looking at the coefficient of $e_{m-n-r,r}$ for $0\leq r\leq k$. For $r=k$, we have that 
$$q^k[m-k]\qfactorial{k}\qhit{k}{m+n-k}{m^k} = \qnumber{m}\qfactorial{k} \qhit{k}{m+n-k}{(m-1)^k},$$
which follows from~\eqref{eq:qhit_mj} using routine simplifications. 

For $r=k-1$, we have that 
\begin{multline}\label{eq:lemma-k-1}
    q^k[m-k]q^{k-1} \qfactorial{k-1}\qnumber{m+n-2k+2} \qhit{k-1}{m+n-k}{m^k} \\
    = \qnumber{m}q^{k-1} \qfactorial{k-1}\qnumber{m+n-2k+2} \qhit{k-1}{m+n-k}{(m-1)^k} - \qnumber{k}\qfactorial{k-1}\qhit{k-1}{m+n-k+1}{m^{k-1}}.
\end{multline}
Using~\eqref{eq:qhit_mj}, we obtain that~\eqref{eq:lemma-k-1} is equivalent to 
\begin{align}\label{eq:lemma-k-1-qnumber}
    q\qnumber{m-k}\qnumber{n-k} = \qnumber{m-k+1}\qnumber{n-k+1} - \qnumber{m+n-2k+1}.
\end{align}

For $0\leq r\leq k-2$, we have that 
\begin{multline}\label{eq:lemma-r}
q^k\qnumber{m-k}q^r \qfactorial{r}\qnumber{m+n-2r} \qhit{r}{m+n-r-1}{m^k} \\ 
= \qnumber{m}q^r \qfactorial{r}\qnumber{m+n-2r} \qhit{r}{m+n-r-1}{(m-1)^k} - \qnumber{k}q^r \qfactorial{r}\qnumber{m+n-2r} \qhit{r}{m+n-r-1}{m^{k-1}}.
\end{multline}
Using~\eqref{eq:qhit_mj}, we obtain that~\eqref{eq:lemma-r} is equivalent to 
\begin{align}\label{eq:lemma-r-qnumber}
q^{k-r}\qnumber{m-k}\qnumber{n-k} = \qnumber{m-r}\qnumber{n-r} - \qnumber{m+n-r-k}\qnumber{k-r},    
\end{align}
which is straightforward to verify by expanding both sides. 
Note that~\eqref{eq:lemma-k-1-qnumber} is a particular case of~\eqref{eq:lemma-r-qnumber} by taking $r=k-1$. 
\end{proof}

We are now ready to prove that the Guay-Paquet's identity and Abreu--Nigro's follow from each other. As a corollary, we obtain a new proof of the latter. The following example illustrates Abreu--Nigro's result. 
\begin{example} \label{ex:tcaseAN}
For $\lambda = (2,1)$ inside a $2\times 3$ board, we have that for $k=2$, 
$\qhit{2}{3}{\lambda}=q^3$, $\qhit{1}{3}{\lambda}=2q^2 + 2q$ (see Figure~\ref{fig:hitex_AN}), and  $\qhit{0}{4}{\lambda}=q^3 + 3q^2 + 3q + 1$. Therefore, 
\[
\csft{2,1}{q} = q^3(1+q) e_{{3,2}}({\bf x})+  q(1+q+q^2)(2q^2+2q) e_{{4,1}}({\bf x})+ (1+q+q^2+q^3)(q^3 + 3q^2 + 3q + 1) e_{{5}}({\bf x}).
\]
\end{example}
\begin{figure}[ht]
\centering
\includegraphics{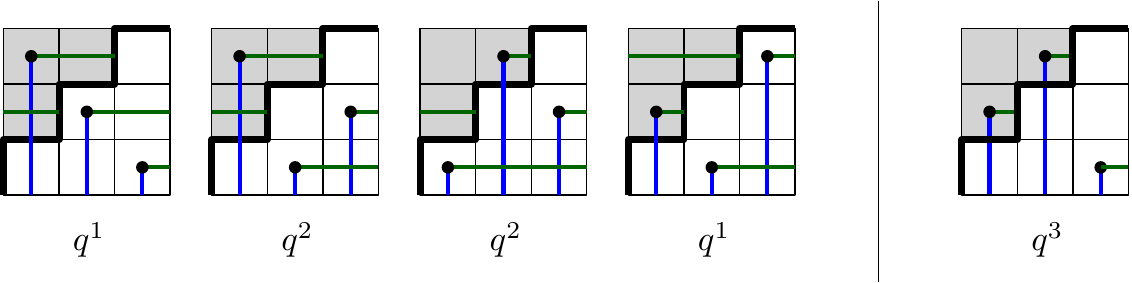}
\caption{Placements of $3$ rooks in $3\times 3$ that ``hit" $(2,1)$ (in gray) once and twice and the
  associated statistic  to each rook placement.}  
\label{fig:hitex_AN}
\end{figure}

\begin{proof}[Proof of Theorem~\ref{AN:generalLambda}]

Applying Lemma~\ref{lem:ANrectangle} to the RHS of the formula in Theorem~\ref{thm:qhitCSFabelian}, we obtain that
\begin{multline*}
    \dfrac{1}{\qfalling{m}{n}} \sum_{j=0}^{n}
\qhit{j}{m,n}{\lambda} \cdot \csft{m^j}{q} 
= \dfrac{1}{\qfalling{m}{n}} \sum_{j=0}^{n}
\qhit{j}{m,n}{\lambda} \Big( \qfactorial{j}\qhit{j}{m+n-j}{m^j}\cdot e_{m+n-j,j}\Big) \\[0.07in]
+\dfrac{1}{\qfalling{m}{n}} \sum_{j=0}^{n}
\qhit{j}{m,n}{\lambda} \Big(  q^r\sum_{r=0}^{j-1}
\qfactorial{r} \qnumber{m+n-2r}\qhit{r}{m+n-r-1}{m^j} \cdot e_{m+n-r,r}\Big).
\end{multline*}
Now, switching the summation order, we have that
\begin{multline*}
    \dfrac{1}{\qfalling{m}{n}} \sum_{j=0}^{n}
\qhit{j}{m,n}{\lambda} \cdot \csft{m^j}{q}
=\sum_{r=0}^n e_{m+n-r,r} \dfrac{1}{\qfalling{m}{n}}\qfactorial{r} \qhit{r}{m+n-r}{m^r}\qhit{r}{m,n}{\lambda} \\
+ \sum_{r=0}^{n-1} e_{m+n-r,r} \dfrac{1}{\qfalling{m}{n}}\Big(q^r\sum_{j=r+1}^n \qfactorial{r} \qnumber{m+n-2r}\qhit{r}{m+n-r-1}{m^j}\qhit{j}{m,n}{\lambda}\Big).
\end{multline*}
Thus, we need to show that for $r=k=\ell(\lambda)$,
\begin{align*}
    \qfalling{m}{n}  \qhit{k}{m+n-k}{\lambda} &= \qhit{k}{m+n-k}{m^k}\qhit{k}{m,n}{\lambda}
    + q^k\sum_{j=k+1}^n  \qnumber{m+n-2k} \qhit{k}{m+n-k-1}{m^j}\qhit{j}{m,n}{\lambda}\\
    &= \qhit{k}{m+n-k}{m^k}\qhit{k}{m,n}{\lambda},
\end{align*}
since $\qhit{j}{m,n}{\lambda}=0$ for $j=k+1,\ldots,n$. We also need to show that for $r<k=\ell(\lambda)$,
\begin{align*}
    \qfalling{m}{n} q^r \qnumber{m+n-2r}\qhit{r}{m+n-r-1}{\lambda}
    &= \qhit{r}{m+n-r}{m^r}\qhit{r}{m,n}{\lambda} \\
    &+ q^r\sum_{j=r+1}^n   \qnumber{m+n-2r}\qhit{r}{m+n-r-1}{m^j}\qhit{j}{m,n}{\lambda}.\nonumber 
\end{align*}
After using~\eqref{eq:qhit_mj}, these two relations are equivalent to the following identities relating $q$-hit numbers of $\lambda$ in square boards and rectangular boards. Thus, the Abreu-Nigro expansion for $\csft{\lambda}{q}$ follows now from Lemma~\ref{prop:qhit-relations} below, which completes the proof. \end{proof}

\begin{lemma}\label{prop:qhit-relations}
Let $\lambda$ be a partition inside an $n\times m$ board and $k=\ell(\lambda)$, then 
\begin{align}
   \qbinom{m-k}{n-k}\qhit{k}{m+n-k}{\lambda} &= q^{k(n-k)}  \qfalling{m+n-2k}{m-k} \qhit{k}{m,n}{\lambda}, \label{eq:keyrel1}
\intertext{
and for $0\leq r< k$, we have}
\label{eq:keyrel2}
   \qbinom{m-r}{n-r}  \qhit{r}{m+n-r-1}{\lambda} &= 
    q^{ r(n-r-1)} \qfalling{m+n-2r-1}{m-r-1} \qhit{r}{m,n}{\lambda} \nonumber\\
    & +  \sum_{j=r+1}^n   q^{r(n-1-j)} \qbinom{j}{r} \dfrac{ \qfalling{m+n-r-j-1}{m-r}  }{\qnumber{n-r}} \qhit{j}{m,n}{\lambda}.
\end{align}

\end{lemma}

\begin{proof}
The first relation follows from a simple combinatorial observation together with Lemma~\ref{lem:q-hit for m,m}. For $k=\ell(\lambda)$, % placed on $\lambda$
we see that the $k$ rooks on the first $k$ rows on the board have to be all inside $\lambda$, and all the cells outside $\lambda$ in these first rows will be empty.
Thus 
$$\qhit{k}{N}{\lambda} = \qhit{k}{k,\lambda_1}{\lambda} q^{k(N-\lambda_1)} [N-k]!.$$
Similarly, by the proof of Lemma~\ref{lemma: rectangle to square our q-hit} we have
$$\qhit{k}{m,n}{\lambda} = \qhit{k}{k,\lambda_1}{\lambda} q^{k(m-\lambda_1)} \qbinom{m+n-2k}{n-k}.$$
Substituting these formulas in each side of~\eqref{eq:keyrel1} we get the desired identity.

For the second relation, we switch gears and use a deletion-contraction relation of $q$-hit numbers (Lemma~\ref{lem: deletion/contration}). The idea is to use induction of the size of $\lambda$, the deletion-contraction relation for $q$-hit numbers, and deduce the identity by matching coefficients at each $q$-hit number. 
\begin{remark}
An alternative proof appears in Section~\ref{app: another proof qhit rel}, where the $q$-hits are expressed in terms of $q$-rooks and matching coefficients at each $q$-rook reduces to $q$-binomial identities. 
\end{remark}

Let us start with the RHS of~\eqref{eq:keyrel2}. Denote by
\begin{align*}
B_r^{m,n}(\lambda) &:=  q^{ r(n-r-1)} \qfalling{m+n-2r-1}{m-r-1} \qhit{r}{m,n}{\lambda}  \\ &+  \sum_{j=r+1}^n   q^{r(n-1-j)} \qbinom{j}{r} \dfrac{ \qfalling{m+n-r-j-1}{m-r}  }{\qnumber{n-r}} \qhit{j}{m,n}{\lambda}.
\end{align*}
If we apply the deletion-contraction relation in Lemma~\ref{lem: deletion/contration} to the $q$-hit numbers, we have that
\begin{align}\label{eq: Bnm-1}
B_r^{m,n}(\lambda) &=  q^{ r(n-r-1)} \qfalling{m+n-2r-1}{m-r-1} \qhit{r}{m,n}{\lambda\backslash e} \nonumber \\
&+ q^{ r(n-r-1)} \qfalling{m+n-2r-1}{m-r-1} q^{|\lambda/e|-|\lambda|+r+m-1}\qhit{r-1}{m-1,n-1}{\lambda/e}\nonumber \\
&- q^{ r(n-r-1)} \qfalling{m+n-2r-1}{m-r-1} q^{|\lambda/e|-|\lambda|+r+m}\qhit{r}{m-1,n-1}{\lambda/e} \nonumber\\
&+ \sum_{j=r+1}^n   q^{r(n-1-j)} \qbinom{j}{r} \dfrac{ \qfalling{m+n-r-j-1}{m-r}  }{\qnumber{n-r}} \qhit{j}{m,n}{\lambda\backslash e}\nonumber \\
&+ \sum_{j=r+1}^n   q^{r(n-1-j)} \qbinom{j}{r} \dfrac{ \qfalling{m+n-r-j-1}{m-r}  }{\qnumber{n-r}} q^{|\lambda/e|-|\lambda|+j+m-1}\qhit{j-1}{m-1,n-1}{\lambda/e} \nonumber\\
&- \sum_{j=r+1}^n   q^{r(n-1-j)} \qbinom{j}{r} \dfrac{ \qfalling{m+n-r-j-1}{m-r}  }{\qnumber{n-r}} q^{|\lambda/e|-|\lambda|+j+m}\qhit{j}{m-1,n-1}{\lambda/e}.
\end{align}

For the LHS of~\eqref{eq:keyrel2}, we denote $C_r^{m,n}(\lambda) := \qbinom{m-r}{n-r}  \qhit{r}{m+n-r-1}{\lambda}$. 
Then by Corollary~\ref{cor: delention/contraction boards}, 
\begin{align*}
C_r^{m,n}(\lambda) &= \qbinom{m-r}{n-r} \qhit{r}{m+n-r-1}{\lambda\backslash e} \\ &+ \qbinom{m-r}{n-r}  q^{|\lambda/e|-|\lambda|+r+m+n-r-2} \left(\qhit{r-1}{m+n-r-2}{\lambda/e} -q\qhit{r}{m+n-r-2}{\lambda/e}\right).
\end{align*}
That is, we have the following deletion-contraction relation for the $C$'s:
\begin{align}\label{eq: Cnm-1}
 C_r^{m,n}(\lambda) =  C_r^{m,n}(\lambda\backslash e) 
+ q^{|\lambda/e|-|\lambda|+m+n-2}\left( C_{r-1}^{m-1,n-1}(\lambda/e) -
q\dfrac{\qnumber{m-r}}{\qnumber{m-n}} C_{r}^{m-1,n}(\lambda/e) \right). 
\end{align}
Now, we want to compare the two expressions in~\eqref{eq: Bnm-1} and~\eqref{eq: Cnm-1}. 
By \emph{inductive hypothesis}, 
\begin{align*}
  C_r^{m,n}(\lambda\backslash e) &= q^{r(n-r-1)}\qfalling{m+n-2r-1}{m-r-1}\qhit{r}{m,n}{\lambda\backslash e}  \\
  & +  \sum_{j=r+1}^n   q^{r(n-1-j)} \qbinom{j}{r} \dfrac{ \qfalling{m+n-r-j-1}{m-r}  }{\qnumber{n-r}} \qhit{j}{m,n}{\lambda\backslash e} = B_r^{m,n}(\lambda\backslash e). 
 \end{align*}
 These terms appear both in~\eqref{eq: Bnm-1} and in~\eqref{eq: Cnm-1}, and so they cancel. We have left to show that 
 \begin{multline*}
q^{|\lambda/e|-|\lambda|+m+n-2}\left( C_{r-1}^{m-1,n-1}(\lambda/e) -
q\dfrac{\qnumber{m-r}}{\qnumber{m-n}} C_{r}^{m-1,n}(\lambda/e) \right) \\
= q^{ r(n-r-1)} \qfalling{m+n-2r-1}{m-r-1} q^{|\lambda/e|-|\lambda|+r+m-1}\qhit{r-1}{m-1,n-1}{\lambda/e}\nonumber \\
- q^{ r(n-r-1)} \qfalling{m+n-2r-1}{m-r-1} q^{|\lambda/e|-|\lambda|+r+m}\qhit{r}{m-1,n-1}{\lambda/e} \nonumber\\
+ \sum_{j=r+1}^n   q^{r(n-1-j)} \qbinom{j}{r} \dfrac{ \qfalling{m+n-r-j-1}{m-r}  }{\qnumber{n-r}} q^{|\lambda/e|-|\lambda|+j+m-1}\qhit{j-1}{m-1,n-1}{\lambda/e} \nonumber\\
- \sum_{j=r+1}^n   q^{r(n-1-j)} \qbinom{j}{r} \dfrac{ \qfalling{m+n-r-j-1}{m-r}  }{\qnumber{n-r}} q^{|\lambda/e|-|\lambda|+j+m}\qhit{j}{m-1,n-1}{\lambda/e}.
 \end{multline*}
 This last equation simplify to the following identity:
   \begin{multline}\label{eq: Bnm-Cnm}
q^{n-2}\left( C_{r-1}^{m-1,n-1}(\lambda/e) -
q\dfrac{\qnumber{m-r}}{\qnumber{m-n}} C_{r}^{m-1,n}(\lambda/e) \right) \\
= q^{ r(n-r)-1} \qfalling{m+n-2r-1}{m-r-1} \left(\qhit{r-1}{m-1,n-1}{\lambda/e}-q\qhit{r}{m-1,n-1}{\lambda/e} \right)\\
+ \sum_{j=r+1}^n   q^{r(n-1-j)+j-1} \qbinom{j}{r} \dfrac{ \qfalling{m+n-r-j-1}{m-r}  }{\qnumber{n-r}}\left(\qhit{j-1}{m-1,n-1}{\lambda/e} - q\qhit{j}{m-1,n-1}{\lambda/e}\right),
 \end{multline}
In order to show~\eqref{eq: Bnm-Cnm}, we apply first the \emph{inductive hypothesis} in the $C$'s together with the relation in Corollary~\ref{lemma:remove column} to obtain their expansion in terms of $q$-hit numbers of the form $\qhit{j}{m-1,n-1}{\lambda/e}$:
 \begin{align*}
 q^{n-2}C_{r-1}^{m-1,n-1}(\lambda/e) &= q^{ (r-1)(n-r-1)+n-2} \qfalling{m+n-2r-1}{m-r-1} \qhit{r-1}{m-1,n-1}{\lambda/e} \\
    & +  \sum_{j=r}^{n-1}   q^{(r-1)(n-2-j)+n-2} \qbinom{j}{r-1} \dfrac{ \qfalling{m+n-r-j-2}{m-r}  }{\qnumber{n-r}} \qhit{j}{m-1,n-1}{\lambda/e},   \\
 q^{n-1}\dfrac{\qnumber{m-r}}{\qnumber{m-n}}C_{r}^{m-1,n}(\lambda/e)   &= q^{ r(n-r-1)+n-1}\dfrac{\qnumber{m-r}}{\qnumber{m-n}} \qfalling{m+n-2r-2}{m-r-2} \qhit{r}{m-1,n}{\lambda/e} \nonumber\\
    & +  \sum_{j=r+1}^{n}   q^{r(n-1-j)+n-1} \qbinom{j}{r} \dfrac{\qnumber{m-r}}{\qnumber{m-n}}\dfrac{ \qfalling{m+n-r-j-2}{m-r-1}  }{\qnumber{n-r}} \qhit{j}{m-1,n}{\lambda/e}.
 \end{align*}
 
Finally, we compare the coefficients in~\eqref{eq: Bnm-Cnm}. 
For $\qhit{r-1}{m-1,n-1}{\lambda/e}$, we have 
\begin{align*}
    LHS &= q^{ (r-1)(n-r-1)+n-2} \qfalling{m+n-2r-1}{m-r-1} = RHS.
\end{align*}

For $\qhit{r}{m-1,n-1}{\lambda/e}$, we have
\begin{align*}
    LHS &= q^{(r-1)(n-2-r)+n-2} \qbinom{r}{r-1} \dfrac{ \qfalling{m+n-r-r-2}{m-r}  }{\qnumber{n-r}} - q^{r(n-r-1)+n-1}\qnumber{m-r} \qfalling{m+n-2r-2}{m-r-2}, \\
    RHS &= -q^{ r(n-r)} \qfalling{m+n-2r-1}{m-r-1} + q^{r(n-2-r)+r} \qbinom{r+1}{r} \dfrac{ \qfalling{m+n-2r-2}{m-r}  }{\qnumber{n-r}},
\end{align*}
which simplifies to the relation $\qnumber{m+n-2r-1} = q^{n-r-1}\qnumber{m-r} +\qnumber{n-r-1}$. This last relation is a straightforward verification by expanding both sides.

For $\qhit{j}{m-1,n-1}{\lambda/e}$, with $r+1\leq j\leq n-1$,
\begin{align*}
    LHS &=q^{(r-1)(n-2-j)+n-2} \qbinom{j}{r-1} \dfrac{ \qfalling{m+n-r-j-2}{m-r}  }{\qnumber{n-r}}  \\ &- 
    q^{r(n-1-j)+n-1} \qbinom{j}{r} \qnumber{m-r}\dfrac{ \qfalling{m+n-r-j-2}{m-r-1}  }{\qnumber{n-r}},\\
    RHS &= q^{r(n-2-j)+j} \qbinom{j+1}{r} \dfrac{ \qfalling{m+n-r-j-2}{m-r}  }{\qnumber{n-r}} - q^{r(n-1-j)+j} \qbinom{j}{r} \dfrac{ \qfalling{m+n-r-j-1}{m-r}  }{\qnumber{n-r}},
\end{align*}
which simplifies to the relation $\qnumber{m+n-r-j-1} = q^{n-j-1}\qnumber{m-r} +\qnumber{n-j-1}$. This last relation is a straightforward verification by expanding both sides.

\end{proof}

\section{Variations and applications} \label{sec: applications and variations}

\subsection{Theorems~\ref{AN:generalLambda} and~\ref{thm:qhitCSFabelian} in terms of unicellular LLT polynomials}\label{subsec:unicellular}

Chromatic symmetric functions of Dyck paths are related to {\em unicellular LLT polynomials} \cite{LLT} that can be defined as follows. For a Dyck path $d$, let $G(d)$ be the associated graph and denote 
\begin{align*}
\LLT_{G(d)}({\bf x},q) := \sum_{\kappa: V(G(d))\to \mathbb{P}} q^{\asc(\kappa)}{\bf x}^{\kappa},
\end{align*}
where the sum is over all vertex colorings $\kappa$ of $G(d)$ and $\asc(\kappa)$ is the same as in the definition of $\csft{G}{q}$.

The function $\LLT_{G(d)}({\bf x},q)$ is actually symmetric (see~\cite[Sec. 3.1]{AP}) and the symmetric functions $X_{G(d)}({\bf x},q)$ and $\LLT_d({\bf x},q)$ are related via a plethystic substitution discovered independently by Carlsson--Mellit~\cite[Prop. 3.4]{CM} and Guay-Paquet~\cite[Lemma 172]{MGP2}:
\[
\LLT_{G(d)}({\bf x},q) = (q-1)^n X_{G(d)}[{\bf x}/(q-1),q],
\]
where $n$ is the size of the Dyck path. As a consequence of this connection, any linear relation in one family implies the same relation in the other one. Since Theorems~\ref{AN:generalLambda} and~\ref{thm:qhitCSFabelian} yield a linear relation among certain chromatic symmetric functions, we immediately obtain the same relations for the corresponding unicellular LLT polynomials.

\begin{corollary}
Let $\lambda$ be partition inside an $n\times m$ board with $\ell(\lambda) = k \leq \lambda_1$. Then
\begin{align*}
\LLT_{G(\lambda)}({\bf x},q) &=
\frac{1}{[m+n-k]!}\qhit{k}{m+n-k}{\lambda}\cdot \LLT_{K(m+n-k,k)}({\bf x},q) \\ &+ \sum_{j=0}^{k-1}
q^j \frac{\qnumber{m+n-2j}}{[m+n-j]!} \qhit{j}{m+n-j-1}{\lambda} \cdot \LLT_{K(m+n-j,j)}({\bf x},q),
\end{align*}
where $K(a,b)$ is the disjoint union of complete graphs on vertices $\{1,\ldots,a\}$ and $\{a+1,\ldots,a+b\}$.
\end{corollary}

\begin{proof}
Since $X_{K(a,b)}({\bf x},q)=[a]![b]! e_{a,b}$, then Theorem~\ref{AN:generalLambda} yields a linear relation among $X_{G(\lambda)}({\bf x},q)$, $X_{K(m+n,0)}({\bf x},q),\ldots, X_{K(m+n-k,k)}({\bf x},q)$. The result then follows from the fact that 
every linear relation among a set of chromatic symmetric functions of Dyck paths has a corresponding relation among unicellular LLT polynomials~\cite[Prop. 55]{AP}.
\end{proof}

\begin{corollary}
Let $\lambda$ be a partition inside an $n\times m$ ($n\leq m$) board. Then
\[
\LLT_{G(\lambda)}({\bf x},q) = \frac{1}{[m]_n} \sum_{j=0}^n \qhit{j}{m,n}{\lambda} \cdot \LLT_{G(m^j)}({\bf x},q).
\]
\end{corollary}

\begin{proof}
Theorem~\ref{thm:qhitCSFabelian} yields a linear relation among $X_{\lambda}({\bf x},q)$, $X_{m^0}({\bf x},q),\ldots, X_{m^n}({\bf x},q)$. The result then follows by the same fact as in the proof above. 
\end{proof}

\subsection{The staircase basis}
Let $V^{m,n}:=\text{span}_{\mathbb{Q}(q)}\{ X_{\lambda} \mid \lambda \text{ inside an } n\times m \text{ board}\}$. 
Consider the set of chromatic symmetric functions given by rectangular shapes $\mathcal{R}^{m,n}:=\left\{\csft{m^j}{q}\right\}_{j=0}^n$. This set is actually a basis for $V^{m,n}$.
\begin{corollary} \label{cor: rectangle is basis}
The set $\mathcal{R}^{m,n}$ forms a basis for the $\mathbb{Q}(q)$-vector space $V^{m,n}$.
\end{corollary}

To prove this result we need the following result.
\begin{proposition} \label{prop: Schur hit expansion}
Let $\lambda$ be a partition inside the $n\times m$ board and let  $k=\ell(\lambda)\leq \lambda_1$
\begin{multline}\label{eq:Schur rectangle}
\csft{\lambda}{q} = [k]! \qhit{k}{m+n-k}{\lambda}s_{2^k1^{m+n-2k}}({\bf x}) \\
+ \sum_{i=0}^{k-1} s_{2^i1^{m+n-2i}}({\bf x}) \Big( [k]!\qhit{k}{m+n-k}{\lambda} + \sum_{j\geq i}^{k-1} q^j[j]![m+n-2j] \qhit{j}{m+n-j-1}{\lambda}\Big).
\end{multline}
\end{proposition}

\begin{proof}
This expansion follows by combining expressing the elementary symmetric functions in the Schur basis  $e_{n-j,j}=\sum_{i=0}^j s_{2^i1^{n-2i}}$ in Theorem~\ref{AN:generalLambda}.
\end{proof}

\begin{proof}[Proof of Corollary~\ref{cor: rectangle is basis}]
By Theorem~\ref{thm:qhitCSFabelian} we have that $\mathcal{R}^{m,n}$ span $V^{m,n}$. 

By Theorem~\ref{AN:generalLambda} the leading term $\csft{m^j}{q}$ in the $e$-basis is $\qhit{j}{m+n-j}{m^j}e_{m+n-j,j}$. Since $j\leq m$ we can certainly place $j$ rooks on $m^j$ and see that  $\qhit{j}{m+n-j}{m^j}\neq 0$. Note that the full formula is given in Proposition~\ref{prop:qhit_rectangles}. Thus the transition matrix between $\mathcal{R}^{m,n}$  and the $e_{m+n-j,j}$'s is upper triangular with nonzero diagonal entries, so they are linearly independent and form a basis.
\end{proof}

Guay-Paquet considered another basis for the space $V^{m,n}$. This basis is called the \emph{staircase basis} since it is indexed by the staircase partitions, which are partitions inside an $n\times m$ board of the form $\delta_j = (j,j-1,\ldots,1)$, for $j=0,\ldots,n$. Let $\mathcal{S}^{m,n}:=\left\{ \csft{\delta_j}{q}\right\}_{j=0}^n$.

\begin{proposition}[\cite{MGP_LR}] \label{prop: staircase is basis}
The set $\mathcal{S}^{m,n}$ forms a basis for the $\mathbb{Q}(q)$-vector space  $V^{m,n}$.
\end{proposition}

\begin{proof}

We look at the expansion in Proposition~\ref{AN:generalLambda} for $\delta_j$. In this case, %$\qhit{j}{m,n}{\delta_i}=0$ for $j>i$ because $\ell(\delta_i)=i$, and
$\qhit{j}{m+n-j}{\delta_j} \neq 0$ because we can place $j$ rooks on the main diagonal of $\delta_j$ and place $m+n-2j$ remaining rooks outside $\delta_j$ (for instance, in the main diagonal of the square board). Therefore, the leading term of $\csft{\delta_j}{q}$ is $\qhit{j}{m+n-j}{\delta_j}e_{m+n-j,j}\neq 0$. That is, $\mathcal{S}^{m,n}$'s transfer matrix with the $V^{m,n}$ basis given by $\{e_{m+n-j,j}\}$ is upper triangular with nonzero diagonal and  hence is also a basis.
\end{proof}

Since the coefficients of $\csft{\lambda}{q}$ in the $\mathcal{R}^{m,n}$ basis involve $q$-hit numbers, we also want to study the coefficients appearing in the decomposition of $\csft{\lambda}{q}$ in this new basis $\mathcal{S}^{m,n}$. That is, we want to understand the coefficients $\ac{j}{m,n}{\lambda}{q}$ in the expression
\begin{equation} \label{eq:def aj}
\csft{\lambda}{q} = \sum_{j=0}^n \ac{j}{m,n}{\lambda}{q} \cdot \csft{\delta_j}{q}. 
\end{equation}

Calculations suggest that up to a sign and a power of $q$, these coefficients are in $\mathbb{N}[q]$.
\begin{conjecture} \label{conj:change of basis}
Let $\lambda$ be a partition inside an $n\times m$ board ($n\leq m$) and $j$ be fixed, then $ \ac{j}{m,n}{\lambda}{q} $ is a Laurent polynomial in $q$ whose coefficients are integers of the same sign.\footnote{This conjecture was proved by Nadeau--Tewari \cite[Thm. 3.5]{NT}.}
\end{conjecture}

\begin{example}
For $\lambda = 3$ inside a $2\times 3$ board, we have 
\[
q^2 \csft{3}{q} = -(q+1) \csft{\varnothing}{q} +(q^2+q+1) \csft{1}{q} + 0\cdot \csft{21}{q}.
\]
\begin{figure}[h]
    \centering
    \includegraphics{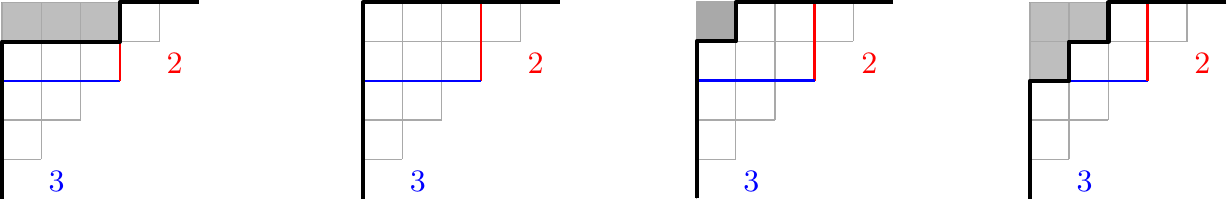}
    \caption{Left: the path for $\lambda=3$ inside a $2\times 3$ board. Right: the three staircase paths $\delta_0,\delta_1,\delta_2$ for $m=3$ and $n=2$.}
    \label{fig:staircase}
\end{figure}
\end{example}

A possible approach to Conjecture~\ref{conj:change of basis} is looking at the change of bases between the rectangular basis $\mathcal{R}^{m,n}$ and the staircase basis $\mathcal{S}^{m,n}$. Let us start showing that the coefficients of the change of bases determine $\ac{j}{m,n}{\lambda}{q}$.
\begin{proposition}\label{prop: a lambda in a mi's}
$\displaystyle{\ac{j}{m,n}{\lambda}{q}  = \frac{1}{[m]_n}\sum_{i=0}^n \qhit{i}{m,n}{\lambda} \ac{j}{m,n}{m^i}{q}}$.
\end{proposition}

\begin{proof}
This follows from basic linear algebra. By Theorem~\ref{thm:qhitCSFabelian} and~\eqref{eq:def aj}, we have that
\begin{align*}
\csft{\lambda}{q} &= \frac{1}{[m]_n}\sum_{k=0}^n \qhit{k}{m,n}{\lambda} \cdot \csft{m^k}{q}\\
&= \frac{1}{[m]_n}\sum_{k=0}^n \qhit{k}{m,n}{\lambda} \sum_{j=0}^n \ac{j}{m,n}{m^k}{q} \csft{\delta_j}{q}.\\
\intertext{Next, we exchange the order of summation  to obtain}
\csft{\lambda}{q} &=\sum_{j=0}^n  \csft{\delta_j}{q} \frac{1}{[m]_n}\sum_{k=0}^n \qhit{k}{m,n}{\lambda}\ac{j}{m,n}{m^k}{q}. %= \sum_{j=0}^n \ac{j}{m,n}{m^k}{q} \csft{\delta_j}{q}.
\end{align*}
The result then follows by extracting the coefficient of $\csft{\delta_j}{q}$ on both sides above.
\end{proof}
This means that in order to prove Conjecture~\ref{conj:change of basis} it suffices to verify it for $\ac{j}{m,n}{m^k}{q}$. 
\begin{remark}\label{remark: inverse matrices}
We have the following $(n+1)\times (n+1)$ matrix and its inverse
\[
\mathcal{A}:=\left(\ac{j}{m,n}{m^k}{q}\right)_{0\leq j,k\leq n}, \quad   %is the matrix 
%given by the $q$-hit numbers of shape $\delta_i$, 
\mathcal{H}:=\mathcal{A}^{-1}= \dfrac{1}{[m]_n}\left(\qhit{j}{m,n}{\delta_k}\right)_{0\leq j,k\leq n}.
\]
Moreover $\qhit{j}{n,n}{\delta_n}=A_{n,j+1}(q)$ equals  the  $(\maj,\des)$  \textbf{$q$-Eulerian polynomials} (see \cite{shareshianwachs2007}) defined by $$\displaystyle{A_{n,j+1}(q):=\sum_{\substack{w\in \mathfrak{S_n} \\ \des(w)=j}} q^{\maj(w)}},$$ 
where $\des(w)$ and $\maj(w)$ are the number of {\em descents} and the {\em major index} of $w$ (see \cite[Eq. I.1,I.2]{GR} and Appendix~\ref{app: relation}). Also, one can check that the specialization 
$$\left.\qhit{j}{n,n}{\delta_{n-r}}\right|_{q=1} = A_{n,j}^{(r)},$$
where the $A^{(r)}_{n,j}$ are the \textbf{$r$-excedence numbers} that count the number of permutations in $\mathfrak{S}_n$ with $j$ $r$-excendences $\{i \mid w(i) \geq i+r\}$ (see~\cite{foata2006theorie,elizalde2020descents}).
\end{remark}

Note that the coefficients $a_j^{m,n}(\lambda,q)$ add up to one.

\begin{proposition}
Let $\lambda$ be a partition inside an $n\times m$ ($n\leq m$) board. Then $\displaystyle{\sum_{j=0}^n \ac{j}{m,n}{\lambda}{q}=1}$.
\end{proposition}
\begin{proof}

By Proposition~\ref{prop: a lambda in a mi's} and Corollary~\ref{cor: sum of hits} it is enough to show the result for the $a_j^{m,n}(m^k,q)$. That is, to show that the columns of the matrix $\mathcal{A}$ from Remark~\ref{remark: inverse matrices} add up to one. This follows from basic linear algebra since the columns of the inverse matrix $\mathcal{A}^{-1}=\mathcal{H}$ also add up to one by Corollary~\ref{cor: sum of hits}.
\end{proof}

\subsection{Theorem~\ref{thm:qhitCSFabelian} in terms of $q$-rook numbers}

Theorem~\ref{thm:qhitCSFabelian} expresses $X_{\lambda}({\bf x},q)$ as a linear combination of the functions $X_{m^j}({\bf x},q)$ where the coefficients are normalized $q$-hit numbers. Since the latter are defined in~\eqref{eq: hit in terms of rs} in terms of $q$-rook numbers it is natural to give an expression for $X_{\lambda}({\bf x},q)$ involving $q$-rook numbers $R_i(\lambda)$.

\begin{definition}
For non-negative integers $m$, $n$, and $i$, with $m\geq n$ and $0\leq i\leq n$,  define
\[
Y^{m,n}_i({\bf x},q) := \sum_{k=0}^i(-1)^{i-k}\qbinom{i}{k}q^{\binom{k}{2}}X_{m^k}({\bf x},q).
\]
\end{definition}

\begin{corollary}
Let $\lambda$ be a partition inside an $n\times m$ board ($n\leq m$). Then
\[
X_{\lambda}({\bf x},q) = \dfrac{1}{\qfactorial{m}}\sum_{i=0}^n \qfactorial{m-i} q^{mi-\binom{i}{2}-|\lambda|}\qrook{i}{\lambda}\cdot Y^{m,n}_i({\bf x},q).
\]
\end{corollary}

\begin{proof}
We use~\eqref{eq: hit in terms of rs} in Theorem~\ref{thm:qhitCSFabelian} and change the order to summation to obtain:
\begin{align*}
    X_{\lambda}({\bf x},q) &= 
    \frac{1}{\qfalling{m}{n}} \sum_{k=0}^n \dfrac{q^{\binom{k}{2}-|\lambda|}}{\qfactorial{m-n}}\sum_{i=k}^n \qfactorial{m-i}\qbinom{i}{k}(-1)^{i+k}q^{mi-\binom{i}{2}}\qrook{i}{\lambda} X_{m^k}({\bf x},q) \nonumber \\
    &= \dfrac{1}{\qfactorial{m}}\sum_{i=0}^n \qfactorial{m-i} q^{mi-\binom{i}{2}-|\lambda|}\qrook{i}{\lambda}\cdot \left( \sum_{k=0}^i(-1)^{i-k}\qbinom{i}{k}q^{\binom{k}{2}}X_{m^k}({\bf x},q)\right).
\end{align*}
The result follows by noting that the sum in parenthesis on the RHS is exactly ${Y_i({\bf x},q)}$. 
\end{proof}

The following example shows that the $Y^{m,n}_i({\bf x},q)$ are not $m$-positive. 
\begin{example}
For $m=3$, $n=2$ and $i=1$, we have that $$Y_1^{3,2}({\bf x},q) = X_{3^1}({\bf x},q) - X_{3^0}({\bf x},q).$$
Expanding in the $m$-basis, we have that 
\begin{align*}
Y_1^{3,2}({\bf x},q) &= \left(-q^{10} - 4 q^{9} - 9 q^{8} - 14 q^{7} - 12 q^{6} - 2 q^{5} + 11 q^{4} + 16 q^{3} + 11 q^{2} + 4 q\right)m_{1,1,1,1,1} \\
&+ \left(q^{6} + 3 q^{5} + 5 q^{4} + 5 q^{3} + 3 q^{2} + q\right)m_{2,1,1,1}.\intertext{
However, in the $e$-basis we have that} 
%Y_1^{3,2}({\bf x},q) 
&= \left(q^{6} + 3 q^{5} + 5 q^{4} + 5 q^{3} + 3 q^{2} + q\right)e_{4,1} \\ &- \left(q^{10} + 4 q^{9} + 9 q^{8} + 14 q^{7} + 17 q^{6} 
+17 q^{5} + 14 q^{4} + 9 q^{3} + 4 q^{2} + q\right)e_{5}
\end{align*}
\end{example}

As noted by John Shareshian (private communication), the expansion of the $Y_i^{m,n}$ in the elementary basis in the previous example is a product of $q$-numbers. This is always the case as the next result shows.

\begin{proposition}
Consider the functions $Y_i({\bf x},q)$, then
\[
Y_i({\bf x},q)=q^{in}[m]![i]! [m+n-2i-1]_{n-i-1} \sum_{r=0}^i (-1)^{i+r} q^{\frac{(r-2i)(r+1)}{2}} [m+n-2r]\qbinom{m+n-r-i-1}{i-r} e_{m+n-r,r}.
\]
\end{proposition}

\begin{proof}
By \cite[\S 4.4.2]{NT}, the Abreu--Nigro expansion of $X_{m^k}({\bf x},q)$ in Lemma~\ref{lem:ANrectangle} can be  written compactly as
\[
X_{m^k}({\bf x},q) = \sum_{j=0}^k q^j [j]![m+n-2j] \qhit{j}{m+n-j-1}{m^k} \cdot e_{m+n-j,j}.
\]
Thus the coefficient $c_{r}(q)$ of the elementary basis $e_{m+n-r,r}$ in $Y_i({\bf x},q)$ equals
\begin{equation*}
c_r(q) = \sum_{k=r}^i (-1)^{i-k} \qbinom{i}{k} q^{\binom{k}{2}+r} [r]![m+n-2r] \qhit{r}{m+n-r-1}{m^k}.
\end{equation*}
Next, we use the explicit expression for the $q$-hit numbers of rectangles in Proposition~\ref{prop:qhit_rectangles} and standard manipulations to obtain 
\begin{multline*}
   c_r(q) = (-1)^{i+r} [i]! [m+n-r-i-1]![m+n-2r] [m]_r q^{nr-\binom{r+1}{2}} \times \\
   \times \sum_{k=r}^i (-1)^{k-r} q^{\binom{k-r}{2}} \qbinom{n-r-1}{k-r}\qbinom{m+n-r-1-k}{i-k}.
\end{multline*}
Lastly, we simplify the sum on the RHS above using the following identity of $q$-binomials 
\[
\sum_{a+b=c} (-1)^a q^{\binom{a}{2}}\qbinom{A}{a}\qbinom{B+b}{b}=q^{Ac}\qbinom{B-A+c}{c}\] 
for $A=n-r-1$, $B=m+n-r-1-i$, and $c=i-r$, and do standard manipulations, to obtain the desired result.
\end{proof}

\section{Final Remarks}\label{sec: final remarks}

\subsection{A tale of two variants of $q$-hit numbers}\label{subsec:tale}

To our surprise, the $q$-hit numbers appearing in Theorems~\ref{AN:generalLambda} and Theorem~\ref{thm:qhitCSFabelian} are not exactly the Garsia--Remmel $q$-hit numbers denoted by $\nqhit{j}{n}{\lambda}$ but instead, they are off by a power of $q$ (see Proposition~\ref{prop: difference GR and our qhit}). Thus, these two variants of $q$-hit numbers satisfy different versions of deletion-contraction (see Appendix~\ref{app: deletion-contraction}).

While working on this project, we have encountered two different variants of Dworkin's statistic for the Garsia--Remmel $q$-hit numbers $\nqhit{n}{j}{\lambda}$ for $\lambda \subset n\times n$. In \cite[Sec. 7, Fig. 3]{D} Dworkin gives a statistic with a rule that is the transpose of the statistic in Definition~\ref{def: our stat qhit}. However, this statistic actually yields our $q$-hit numbers $\qhit{j}{n}{\lambda}$ (this can be seen from the change of basis to $q$-rook numbers and because the latter stay invariant under conjugation) and {\em not} the Garsia--Remmel $q$-hit numbers $\nqhit{j}{n}{\lambda}$ as claimed in \cite{D}. In \cite[Fig. 5]{HR}, Haglund and Remmel give a statistic similar to Dworkin that they attribute to him (see Definition~\ref{def: Dworkin stat qhit}) that {\em does} yield the Garsia--Remmel $q$-hit numbers. The authors in \cite{HR} give a weight-preserving bijection between their version of Dworkin's statistic and Haglund's statistic for $\nqhit{j}{n}{\lambda}$ from \cite{H} thus proving the validity of their version of Dworkin's statistic.

 See Appendices~\ref{app: Dworkin} and \ref{app: relation} for more details on the  Garsia--Remmel $q$-hit numbers and their relation to our $q$-hit numbers. And see Appendix~\ref{app: proof statistic} for a proof of Theorem~\ref{thm: qhit statistic rectangular board} which shows the statistic in Definition~\ref{def: our stat qhit} that yields the $q$-hit numbers $\qhit{m,n}{j}{\lambda}$. 
 
 For a recent $q$-analogue of hit numbers for general boards, not just {\em Ferrers boards}, see \cite{LM17}.

\subsection{A conjectured deletion-contraction relation for $q$-hit numbers}

We use the deletion-contraction of $q$-hit numbers (Lemma~\ref{lem: deletion/contration}) in our proof of Theorem~\ref{AN:generalLambda}. It appears that the $q$-hit numbers satisfy a similar deletion-contraction relation with simpler powers of $q$. For more details in another deletion-contraction relation of $q$-hit numbers see Appendix~\ref{app: deletion-contraction}.

\begin{conjecture} \label{lemma:del con hits}
Let $\lambda$ be a partition inside an $n\times m$ board and $e$ be an outer corner of $\lambda$. Then we have the following recursion: $\qhit{j}{m,n}{\varnothing} = \qfalling{m}{n} \delta_{j,0}$ and 
\begin{equation*} %\label{eq:delcon_rect}
q\qhit{j}{m,n}{\lambda} =  \qhit{j}{m,n}{\lambda\backslash e}
+ q^m \qhit{j-1}{m-1,n-1}{\lambda/e} - \qhit{j}{m-1,n-1}{\lambda/e}.\footnote{This conjecture was proved by Nadeau--Tewari (private communication).}
\end{equation*}
\end{conjecture}

\subsection{Combinatorial proof of Theorem~\ref{thm:qhitCSFabelian}}

Guay-Paquet's proof of  Theorem~\ref{thm:qhitCSFabelian} sketched in \cite{MGP_LR} is based on the idea of dual basis from linear algebra. He shows that the vector space $V^{m,n}$, together with the basis $\mathcal{R}^{m,n}$, has a dual vector space $V^*_{m,n}:=\text{span}_{\mathbb{Q}(q)}\big( P(x;\lambda)/[m]_n \mid \lambda \subset n\times m \big)$ with dual basis $\{x^i \mid i=0,\ldots n\}$. Now, the dual basis coefficients are given by the normalized $q$-hit numbers $\qhit{i}{m,n}{\lambda}/[m]_n$ as shown (up to normalization) in~\eqref{eq: hit rook change of basis}.

In contrast, our proof of Theorem~\ref{thm:qhitCSFabelian} uses $q$-rook theory, it would be interesting to find a bijective proof of this result relating colorings with rook placements. 

Specifically, we can rewrite the identity as 
\begin{equation*}
    [m]_n \csft{\lambda}{q} = \sum_{k=0}^n \qhit{k}{m,n}{\lambda} \csft{m^k}{q}.
\end{equation*}
Matching monomials in ${\bf x}$, powers of $q$, and interpreting $[m]_n=\qhit{0}{m,n}{\emptyset}$ we are looking for a bijection $\varphi_\nu$ for every $\nu =(2,\ldots,2,1,\ldots)$ as follows.

\emph{From:}

 Pairs of a rook placement $P_0^{m,n}(\emptyset)$ of $n$ rooks on $n \times m$ board with $\inv(P)$ inversions and a proper coloring $\kappa(\lambda,\nu)$ of $G_{\lambda}$ of type $\nu$ and $\asc(\kappa)$ total ascents. 

\emph{Into:}

 Triples $k, P_k^{m,n}(\lambda), \kappa(m^k,\nu)$ consisting of an integer $k\in [0,n]$, a rook placement on $n \times m$ with exactly $k$ rooks inside $\lambda$, and a proper coloring of $G_{m^k}$ of type $\nu$, 

\emph{such that} 
$$\asc(\kappa(m^k,\nu))+\inv(P_k^{m,n}) = \asc(\kappa)+\inv(P).$$

\subsection{Combinatorial proof of Theorem~\ref{AN:generalLambda}}

There are other rules for the elementary basis expansion of $\csft{\lambda}{q}$. In particular, Cho--Huh~\cite{ChoHuh} give an expansion in terms of {\em $P$-tableaux} of shape  $2^j 1^{m+n-2j}$ such that there is no $s \geq j+2$ such that  $(a_{i,1},a_{s,1}) \in \lambda$ for all $i \in \{ \ell+1,\ldots,s-1\}$ (see~\cite[Sec. 6]{ShW2} for details on $P$-tableaux). We denote such set of $P$-tableaux by $\mathcal{T}'((m+n-j,j))$.  Let $$c_j^{m,n}(q):=\sum_{T \in \mathcal{T}'(m+n-j,j)} q^{\inv_{G(\lambda)}(T)}.$$
For the definitions of $P$-tableau and $\inv_G(T)$ see~\cite[Sec. 6]{ShW2}.

 It would be interesting to find a weight-preserving bijection that shows that 
\[
c_j^{m,n}(q) = \begin{cases}
[j]! \qhit{j}{m+n-j}{\lambda} &\text{ if } j=\ell(\lambda),\\
q^j[j]! [m+n-2j] \qhit{j}{m+n-j-1}{\lambda} &\text{ if } j<\ell(\lambda).
\end{cases}
\]

For instance, for the case $j=\ell(\lambda)$, we need to establish a weight-preserving bijection between the rook placements with $j$ rooks inside $\lambda$, together with some labeling of the ones inside $\lambda$ to account for $\qfactorial{j}$, and $\mathcal{T}'((m+n-j,j))$.

\subsection{Extending Theorems~\ref{AN:generalLambda} and \ref{thm:qhitCSFabelian} to bicolored graphs} 

Both Theorem~\ref{AN:generalLambda} and Theorem~\ref{thm:qhitCSFabelian} are $q$-analogues of a special cases of respective results by Stanley--Stembridge \cite[Thm. 4.3]{StSt} and by Guay-Paquet \cite[Prop. 4.1 (iv)]{MGP}  for {\em bicolored graphs}. It would be interesting to find a $q$-analogue of these more general results for bicolored graphs $G$. However, for such graphs $G$ the function $X_G({\bf x},q)$ may not be symmetric.

\subsection{Beyond the abelian case}

We have studied the chromatic symmetric function $\csft{G(d)}{q}$ for Dyck paths $d$ of bounce two, aka the abelian case \cite{HP}. 
Recently Cho--Hong \cite{CHo} verified Conjecture~\ref{conj: StaSteShWa} when $q=1$ for Dyck paths of bounce three.  Their expansion is in terms of certain $P$-tableaux. For these Dyck paths, it would be interesting to find an $e$-expansion involving $q$-rook theory or extending Theorem~\ref{thm:qhitCSFabelian}.

 \bibliographystyle{alpha}
 \bibliography{qhit-biblio-fpsac.bib}

\appendix
\section*{Appendix}\label{sec:appendix}
\renewcommand{\thesubsection}{\Alph{subsection}}
\numberwithin{theorem}{subsection}
\numberwithin{equation}{subsection}

In this appendix we give more details on the two variants of the statistic on the Garsia--Remmel $q$-hit numbers, their relation, and their deletion-contraction relations. 
\subsection{Garsia-Remmel $q$-hit numbers and Dworkin's statistic}\label{app: Dworkin}
We start by defining the original version of the $q$-hit numbers given by Garsia--Remmel that is different than our $q$-hit numbers. Recall that $m\geq n$.

\begin{definition}[\cite{GR}]
For $\lambda$ inside an $n\times n$ board, we define the Garsia-Remmel \emph{$q$-hit polynomial} of $\lambda$ by 
\begin{equation} \label{eq: GR hit rook change of basis}
\sum_{i=0}^n \nqhit{i}{n,n}{\lambda} x^i := \sum_{i=0}^n \qrook{i}{\lambda} \qfactorial{n-i} \prod_{k=n-i+1}^n (x-q^k).
\end{equation}
\end{definition}

Garsia and Remmel~\cite[Theorem 2.1]{GR} showed that $\nqhit{i}{n}{\lambda}:=\nqhit{i}{n,n}{\lambda}$ is a polynomial with nonnegative coefficients and Dworkin~\cite{D} and Haglund~\cite{H} gave different statistics realizing these $q$-hit numbers. We focus on Dworkin's statistic (as presented in \cite{HR}, see Section~\ref{subsec:tale}) since it is very similar to the statistic in Definition~\ref{def: our stat qhit}.

\begin{definition}[Dworkin's statistic for the $q$-hit numbers~\cite{HR}]\label{def: Dworkin stat qhit}
Let $\lambda$ be a partition inside an $n\times m$ board. Given a placement $p$ of $n$ non-attacking rooks on an $n\times n$ board, with exactly $j$ inside $\lambda$, let $\dstat(p)$ be the number of cells $c$ in the $n\times m$ board such that 
\begin{compactitem}
\item[(i)] there is no rook in $c$, 
\item[(ii)] there is no rook below $c$ on the same column, and either, 
\item[(iii)] if  $c$ is in $\lambda$ then the rook on the same row of $c$ is either outside $\lambda$ or else to the left of $c$; or  
\item[(iv)] if $c$ is not in $\lambda$ then the rook on same row of $c$ is not in $\lambda$ and to the left  of $c$. 
\end{compactitem}
\end{definition}

\begin{example}
Consider the partition $\lambda=(4,3,2,2)$ inside a $6\times 6$ board. In Figure~\ref{fig: Dworkin q-hit statistic square}, we present a rook placement $p$ of six rooks on the  $6\times 6$ board with three hits on $\lambda$ where $\dstat(p)=4$.
\end{example}

\begin{figure}
     \centering
     \begin{subfigure}[b]{0.2\textwidth}
         \centering
         \includegraphics{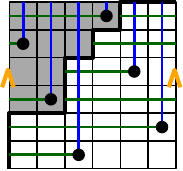}
         
         \caption{}
         \label{fig: Dworkin q-hit statistic square}
         
     \end{subfigure}
     \begin{subfigure}[b]{0.2\textwidth}
         \centering
         \includegraphics{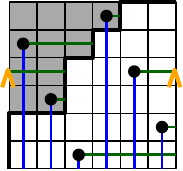}
         \caption{}
         \label{fig: our q-hit statistic square}
     \end{subfigure}
      \begin{subfigure}[b]{0.5\textwidth}
    \includegraphics{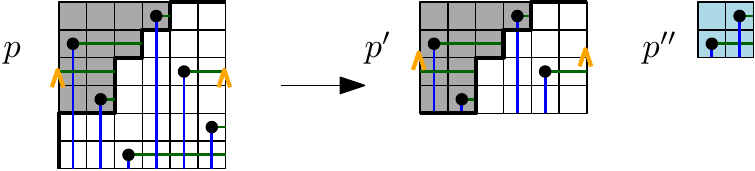}
    \caption{}
    \label{fig:our qhit square to rect}
   \end{subfigure}
    \label{fig:qhit square to rect}
        \caption{(A) For the same rook placement $p$, examples of (A) the statistic  $\dstat(p)$ for the Garsia--Remmel $q$-hit numbers, and (B) of the  statistic $\stat(p)$ for our $q$-hit numbers. Note how the empty cells in the calculation of one statistic correspond to the double crossings in the other. Moreover in each placement $\stat(p)-\cross(p)=j\cdot m-|\lambda|$. (C) Illustration of bijection between rook placement counted in  $\qhit{j}{m,m}{\lambda}$ and rook placements counted in  $\qhit{j}{m,n}{\lambda}$   and $R_{m-n}((m-n)^{m-n})$, respectively.}
        \label{fig: Dworkin vs Our q-hit}
\end{figure}

\begin{theorem}[{Dworkin~\cite{D,HR}}] \label{thm: dworkin rule}
Let $\lambda$ be a partition inside an $n\times n$ board and $j=0,\ldots,n$ then
\[
\nqhit{j}{n}{\lambda} = \sum_p q^{\dstat(p)},
\]
where the sum is over all placements $p$ of $n$ non-attacking rooks on an $n\times n$ board with exactly $j$ rooks inside $\lambda$.
\end{theorem}

\subsection{Relation between $q$-hit numbers $\qhit{j}{n}{\lambda}$ and $\nqhit{j}{n}{\lambda}$}\label{app: relation}

The next result shows that our $q$-hit numbers and the Garsia--Remmel $q$-hit numbers are off by a power of $q$. We show this from the respective definitions of each $q$-hit numbers from the generating polynomials of $q$-rook numbers. We will ultimately show Theorem~\ref{thm: qhit statistic rectangular board} by showing that for a rook placement $p$, the statistics $\stat(p)$ and $\dstat(p)$ are off by the same power (see Lemma~\ref{lem:stat-dstat}).

\begin{proposition} \label{prop: difference GR and our qhit}
\begin{align}\label{eq:change H and tilde H}
    \nqhit{j}{m}{\lambda}:= q^{|\lambda| - jm} \qhit{j}{m}{\lambda}.
\end{align}
\end{proposition}

\begin{proof}
We compare~\eqref{eq: GR hit rook change of basis} and~\eqref{eq: hit rook change of basis}. Substituting $xq^{-m}$ and multiplying by $q^{|\lambda|}$ in~\eqref{eq: hit rook change of basis}, rewriting the factor $q^{mi-\binom{i}{2}}=q^{\sum_{k=0}^{i-1} m - k}$ and rewriting the $q$-Pochhammer symbol we have
\begin{align*}
    \sum_{j=0}^m \qhit{j}{m}{\lambda} q^{|\lambda|} (xq^{-m})^j &=  \sum_{i=0}^m R_i(\lambda)[m-i]! (-1)^i \prod_{k=0}^{i-1} q^{m-k} (1-xq^{-m} q^k)\\
    &= \sum_{i=0}^m R_i(\lambda)[m-i]! \prod_{k=0}^{i-1}(x-q^{m-k}),
\end{align*}
which is the RHS of~\eqref{eq: GR hit rook change of basis} for $m=n$. Comparing coefficients at $x^j$ we obtain the desired identity.
\end{proof}

As a consequence of the symmetry and reciprocity of the Garsia--Remmel $q$-hit numbers we obtain the following result. Given $\lambda$ inside $m\times m$ board, let $\overline{\lambda}$ be the complementary partition of $\lambda$, viewed as a Ferrers board.

\begin{corollary}
\begin{equation}
    \qhit{j}{m}{\lambda} = \nqhit{m-j}{m}{\overline{\lambda}}.
\end{equation}
\end{corollary}

\begin{proof}
The result follows by combining Proposition~\ref{prop: difference GR and our qhit} with both the reciprocity \cite[Lemma 8.19]{D} and the symmetry \cite[Thm. 9.22]{D} of the Garsia--Remmel $q$-hit numbers:
\begin{align*}
    \nqhit{j}{m}{\lambda,q} &=q^{\binom{m}{2}} \nqhit{m-j}{m}{\overline{\lambda},q^{-1}},\\
    q^{\binom{m}{2}+|\lambda|-jm} \nqhit{j}{m}{\lambda,q^{-1}} &= \nqhit{j}{m}{\lambda,q}.
\end{align*}
\end{proof}

\subsection{Proof of Theorem~\ref{thm: qhit statistic rectangular board}} \label{app: proof statistic}

Let $\widehat{H}^{m,n}_j(\lambda)$ be the sum in the RHS of~\eqref{eq: qhit statistic rectangular board}. The next result is an analogue of Lemma~\ref{lem:q-hit for m,m}  for $\widehat{H}^{m,n}_j(\lambda)$.

\begin{lemma}
\label{lemma: rectangle to square our q-hit}
Let $\lambda$ be a partition inside an $n\times m$ board. Then
\[
\widehat{H}^{m,n}_j(\lambda) = \frac{1}{[m-n]!} \widehat{H}^{m,m}_j(\lambda).
\]
\end{lemma}

\begin{proof}
We claim that 
\begin{equation} \label{eq:key rel for rect to square}
\widehat{H}^{m,n}_{j}(\lambda) = \widehat{H}_j^{m,n}(\lambda) \cdot \qrook{m-n}{(m-n)^{m-n}}.
\end{equation} 
The result then follows since by  Proposition~\ref{prop:GR_rook_gen_fun},  $\qrook{m-n}{(m-n)^{m-n}}=[m-n]!$. This $q$-factorial corresponds to the $q^{\stat}$-weighted enumeration of rook placements in a $(m-n)\times (m-n)$ board. Let $p$ be a rook placement on an $m\times m$ board with $j$ rooks inside $\lambda\subset m \times n$ and $p'$ be the placement obtained by restricting $p$ to the top $m$ rows. Then the bottom $m-n$ rows contain $m-n$ rooks outside $\lambda$ and after removing the $n$ columns occupied by rooks from the top $n$ column we obtain a placement $p''$ of $m-n$ rooks on an $(m-n)\times (m-n)$ board. This gives a bijection $p\mapsto (p',p'')$ between the rook placements on the LHS and pairs of rook placements from the RHS of \eqref{eq:key rel for rect to square}. See Figure~\ref{fig:our qhit square to rect} Moreover, the bijection is weight-preserving. That is 
\begin{align*}
    \stat(p) =\stat(p')+\stat(p'') = \stat(p')+\inv(p''),
\end{align*}
where $\inv(p'')$ is the statistic of the $q$-rook numbers. 
This weight-preserving bijection gives the desired result.
\end{proof}

The next lemma shows that for a fixed rook placement the statistics $\stat(\cdot)$ and $\dstat(\cdot)$ are related.

\begin{lemma}\label{lem:stat-dstat}
Let $\lambda$ be a partition inside an $m\times m$ board. Given a placement $p$ of $m$ non-attacking rooks on an $m\times m$ board with $j$ rooks inside $\lambda$ then
\[
\stat(p)-\dstat(p) = j\cdot m - |\lambda|.
\]
\end{lemma}

\begin{example} \label{ex: our q hit and dworkin on a square}
Consider the partition $\lambda=(4,3,2,2)$ inside a $6\times 6$ board. Figure~\ref{fig: our q-hit statistic square} illustrates an example of a placement $p$ of six rooks on the  $6\times 6$ board with $j=3$ hits on $\lambda$ with $\stat(p)=11$.  Figure~\ref{fig: Dworkin q-hit statistic square} illustrates for the same rook placement $p$ that $\dstat(p)=4$. Note that 
\[
\stat(p)-\dstat(p) = 7 = 3\cdot 6 - |\lambda|.
\]
\end{example}

The proof of Lemma~\ref{lem:stat-dstat} is postponed to the end of the section. We now use this lemma to complete the proof of our main result of this appendix.

\begin{proof}[Proof of Theorem~\ref{thm: qhit statistic rectangular board}]
By Lemma~\ref{lem:stat-dstat} we have that 
\[
\nqhit{j}{m,m}{\lambda} = q^{|\lambda|-jm} \widehat{H}^{m,m}_j(\lambda).
\]
Next, by comparing this identity with Proposition~\ref{prop: difference GR and our qhit} we conclude that
$\qhit{j}{m,m}{\lambda}=\widehat{H}^{m,m}_j(\lambda)$.
 Finally, combining this with both Lemma~\ref{lem:q-hit for m,m} and Lemma~\ref{lemma: rectangle to square our q-hit} we conclude that  $\qhit{j}{m,n}{\lambda}=\widehat{H}^{m,n}_j(\lambda)$ as desired.
\end{proof}

The rest of the section is devoted to the proof of Lemma~\ref{lem:stat-dstat}. We need the following definition.

\begin{definition}[crossing statistic for the $q$-hit numbers]\label{def: transpose our stat qhit}
Let $\lambda$ be a partition inside an $m\times m$ board. Given a placements $p$ of $m$ non-attacking rooks on an $n\times m$ board, with exactly $j$ rooks inside $\lambda$, let $\cross(p)$ be the number of cells $c$ in the $m\times m$ board such that  
\begin{compactitem}
\item[(i)] there is no rook in $c$,
\item[(ii)] there is a rook on the same column and below $c$,
\item[(ii)] if $c$ is in $\lambda$ then there is a rook on the same row in $\lambda$ to the right of $c$,
\item[(iii)] if $c$ is not in $\lambda$ then either there is a rook on the same row in $\lambda$ or a rook on the same row to the right of $c$.
\end{compactitem}
In other words, $\cross(p)$ is the number of cells that have double crossings after the rook cancellations used to obtain $\stat(p)$. See Remark~\ref{rem: rook cancellation in cylinder} and Figure~\ref{fig: our q-hit statistic square}.
\end{definition}

\begin{example} \label{ex: crossings for our q hit on a square}
For the rook placement $p$ in Example~\ref{ex: our q hit and dworkin on a square} and Figure~\ref{fig: our q-hit statistic square} we have that $\stat(p)=11$ and $\cross(p)=4$. 
\end{example}

First observe that for a rook placement $p$ on the $m\times m$ board we have that $\dstat(p)=\cross(p)$, since the rays in $\stat$ and $\dstat$ are complementary to each other and the crossings in one directly correspond to the empty boxes in the other. Therefore Lemma~\ref{lem:stat-dstat} follows from the next lemma.

\begin{lemma} 
Let $\lambda$ be a partition inside an $m\times m$ Given a placement $p$ of $m$ non-attacking rooks on an $m\times m$ board with $j$ rooks inside $\lambda$ then
\begin{equation} \label{eq:stat-cross}
\stat(p)-\cross(p) = j\cdot m - |\lambda|.
\end{equation}
\end{lemma}

\begin{proof}
We proceed by induction on $|\lambda|$ for $\lambda \subset m\times m$. When $\lambda =\emptyset$ we only have rook placements for $j=0$, and then it is clear that,
\[
\stat(p)=\dstat(p)=\cross(p)=\inv(w),
\]
the number of inversions of the permutation $w$  corresponding to the rook diagram. Thus the identity \eqref{eq:stat-cross} is verified.

Suppose the identity \eqref{eq:stat-cross} holds for all $|\lambda|\leq N$ and then for any $j=0,\ldots,m$. Let $\nu$ be a partition of $N+1$ and $\nu = \lambda+ e $, where $e$ is a corner cell. Let $p$ be a rook configuration with $j$ rooks in $\nu$, and let $p'$ be the same rook configuration on $\lambda$ (so there are  $j$ or $j-1$ rooks in $\lambda$). 
Note that cell $e$ cannot be empty since there is a rook in its row, which is either in $\nu$, and hence the rook's  ``arm" crosses $e$ or is outside in which case the wrap-around also crosses $e$. We now consider several cases:

\begin{compactitem}
\item[Case 1.] Cell $e$ has a rook and hence the horizontal arm stops at $e$ as the border of $\nu$. Thus there are no crossings in the row of $e$. Suppose there are $k$ empty boxes in the row of $e$, then there are also $m-k-1$ vertical crossings in this row. Now consider $p'$ as a configuration with $j-1$ rooks in $\lambda$. The rook in $e$ is outside $\lambda$ and this time the entire row is crossed, so all empty cells have now a horizontal line and all vertically crossed cells have now a double crossing. Thus 
\[
\qquad \qquad \stat(p') - \cross(p') = \stat(p) -k - (\cross(p) + m-k-1) =  \stat(p) -\cross(p) -m +1.
\]
By induction we have 
\[
\stat(p')-\cross(p')=m(j-1) - |\lambda| = j\dot m - |\nu| -m +1,
\]
and matching sides we obtain the desired identity in this case.

\item[Case 2.] The rook in the row of $e$ is to the left of $e$. Then either $e$ is a double crossing or is only a horizontal crossing. Then in the rook placement $p'$ there is no horizontal line crossing $e$. If $e$ was a double crossing in $p$, then it is neither a double crossing nor empty cell in $p'$, and if $e$ was not a double crossing in $p$ then it became an empty cell in $p'$. In both cases we have
\[
\stat(p')-\cross(p')=\stat(p)-\cross(p)+1.
\]
Since the number of rooks inside $\lambda$ is still $j$ then we have
\[
\stat(p)-\cross(p)= \stat(p')-\cross(p') -1 = j\cdot m - |\lambda|-1 =j\cdot m -|\nu|.
\]
This gives the desired identity in this case.

\item[Case 3.] The rook in the row of $e$ is to right of $e$, so outside $\nu$. Then again there is a horizontal line crossing $e$, so $e$ is either a double crossing in $p$ or neither a double crossing nor an empty cell in $p$. In both cases when we remove $e$ from $\nu$ we either turn the double crossing on $e$ in $p$ to a not a double crossing in $p'$ or from not a double crossing in $e$ to an empty cell in $p'$. Thus, again 
\[
\stat(p')-\cross(p') = \stat(p)-\cross(p) +1.
\]
Since the number of rooks inside $\lambda$ is still $j$ then we have
\[
\stat(p)-\cross(p)= \stat(p')-\cross(p') -1 = j\cdot m - |\lambda|-1 =j\cdot m -|\nu|.
\]
This gives the desired identity in this case.
\end{compactitem}
This completes the proof.
\end{proof}

\subsection{Symmetry of $q$-hit numbers of rectangular boards} \label{app: symmetry qhit}

Since the Garsia--Remmel $q$-hit numbers are symmetric polynomials in $\mathbb{N}[q]$ \cite{GR,D,H}, then so are $\qhit{j}{m,n}{\lambda}$.

\begin{corollary}
The $q$-hit numbers $\qhit{j}{m,n}{\lambda}$ are symmetric  polynomials in $\mathbb{N}[q]$.
\end{corollary}

\begin{proof}
By Theorem~\ref{thm: qhit statistic rectangular board} we have that $\qhit{j}{m,n}{\lambda}$ are in $\mathbb{N}[q]$. By Lemma~\ref{lem:q-hit for m,m} and Proposition~\ref{prop: difference GR and our qhit} we have that 
\[
\qhit{j}{m,n}{\lambda} = \frac{1}{[m-n]!} q^{|jm-\lambda|} \widetilde{H}_j^{m}(\lambda).
\]
Now $[m-n]!$ is a symmetric polynomial in $\mathbb{N}[q]$ and so are the Garsia--Remmel $q$-hit numbers  $\widetilde{H}_j^{m}(\lambda)$ \cite[Sec. 5]{H}. Therefore,  the result follows.
\end{proof}

\subsection{Deletion-contraction for $q$-hit numbers} \label{app: deletion-contraction}

In this section we give a proof of the deletion-contraction relations for the $q$-hit numbers $\qhit{j}{m,n}{\lambda}$ and $\nqhit{j}{m,n}{\lambda}$.

\begin{lemma}[{\cite[Thm. 6.11]{D}}] \label{lemma:del con GR hits}
Let $\lambda$ be a partition inside an $n\times m$ board and $e$ be an outer corner of $\lambda$. Then we have the following recursion:  
\begin{equation*} %\label{eq:delcon_rect}
\nqhit{j}{m,n}{\lambda} =  q\nqhit{j}{m,n}{\lambda\backslash e}
+ \nqhit{j-1}{m-1,n-1}{\lambda/e} - q^m \nqhit{j}{m-1,n-1}{\lambda/e}, \qquad \nqhit{j}{m,n}{\varnothing} = \qfalling{m}{n} \delta_{j,0}.
\end{equation*}
\end{lemma}

\begin{proof}
This follows from the deletion-contraction relation for $q$-rook numbers~\cite[Thm. 6.10]{D} 
\[
\qrook{i}{\lambda} \,=\, q \cdot \qrook{i}{\lambda \backslash e} + \qrook{i-1}{\lambda / e}, \qquad \qrook{i}{\varnothing} = \delta_{i,0},
\]
which follows directly from considering if a placement $p$ of $i$ rooks in $\lambda$ has or not a rook in cell $e$. 
Substituting this rook recursion in~\eqref{eq: hit in terms of rs}, we obtain
\begin{multline*}
    \qhit{k}{m,n}{\lambda} = \dfrac{q^{\binom{k}{2}-|\lambda|}}{\qfactorial{m-n}}\sum_{i=k}^n ( q \cdot \qrook{i}{\lambda \backslash e} + \qrook{i-1}{\lambda / e}) \qfactorial{m-i} \qbinom{i}{k} (-1)^{i+k} q^{mi-\binom{i}{2}}\\
    =  \dfrac{q^{\binom{k}{2}-|\lambda \backslash e|}}{\qfactorial{m-n}}\sum_{i=k}^n \qrook{i}{\lambda \backslash e} \qfactorial{m-i} \qbinom{i}{k} (-1)^{i+k} q^{mi-\binom{i}{2}} + \dfrac{q^{\binom{k}{2}-|\lambda|}}{\qfactorial{m-n}}\sum_{i=k}^n \qrook{i-1}{\lambda / e}\qfactorial{m-i} \qbinom{i}{k} (-1)^{i+k} q^{mi-\binom{i}{2}}.
\end{multline*}
Manipulating the last expression from $q$-rook numbers into $q$-hit numbers, we obtain the following recurrence

\begin{multline*}
   \qhit{k}{m,n}{\lambda} =  \qhit{k}{m,n}{\lambda \backslash e}  + q^{m+k-1 -|\lambda|+|\lambda/e|} \qhit{k-1}{m-1,n-1}{\lambda/e} -q^{k+m -|\lambda|+|\lambda/e|} \qhit{k}{m-1,n-1}{\lambda/e}
\end{multline*}

Now, we use~\eqref{eq:change H and tilde H} to translate this recursion into the recursion for the $\widetilde{H}$'s:
\begin{align*}
     q^{km-|\lambda|} \nqhit{k}{m,n}{\lambda} \\
&=  q^{km - |\lambda\backslash e|}\nqhit{k}{m,n}{\lambda \backslash e}  +  q^{ (k-1)(m-1)-|\lambda/e|}q^{m+k-1 -|\lambda|+|\lambda/e|} \nqhit{k-1}{m-1,n-1}{\lambda/e} \\ 
&- q^{k(m-1)-|\lambda/e|}q^{k+m -|\lambda|+|\lambda/e|} \qhit{k}{m-1,n-1}{\lambda/e},
\end{align*}
which simplifies to the desired recursion.

\end{proof}

\begin{proof}[Proof of Lemma~\ref{lem: deletion/contration}]
Combining together~\eqref{eq:change H and tilde H} and Lemma~\ref{lemma:del con GR hits}, we obtain
\begin{align*}
q^{|\lambda| - jm} \qhit{j}{m,n}{\nu}  &=  q^{|\lambda\backslash e| - jm+1} \qhit{j}{m,n}{\lambda\backslash e} \\
&+ q^{|\lambda/e| - (j-1)(m-1)} \qhit{j-1}{m-1,n-1}{\lambda/e} - q^{|\lambda/e| - j(m-1)+m} \nqhit{j}{m-1,n-1}{\lambda/e}.
\end{align*}
Noticing that $|\lambda\backslash e| + 1 = |\lambda|$ and simplifying the expression we obtain that
\begin{align*} 
\qhit{j}{m,n}{\nu}  &= \qhit{j}{m,n}{\lambda\backslash e}
+ q^{|\lambda/e|-|\lambda|+j+m-1} \qhit{j-1}{m-1,n-1}{\lambda/e} - q^{|\lambda/e|-|\lambda|+j+m} \qhit{j}{m-1,n-1}{\lambda/e} \\
&= \qhit{j}{m,n}{\lambda\backslash e}
+ q^{|\lambda/e|-|\lambda|+j+m-1}\left( \qhit{j-1}{m-1,n-1}{\lambda/e} - q \qhit{j}{m-1,n-1}{\lambda/e}\right).
\end{align*}
\end{proof}

The previous deletion-contraction relation specializes to square boards as follows.
\begin{corollary} \label{cor: delention/contraction boards}
$\qhit{j}{N}{\lambda}  = \qhit{j}{N}{\lambda\backslash e}
+ q^{|\lambda/e|-|\lambda|+j+N-1}\left[ \qhit{j-1}{N-1}{\lambda/e} - q \qhit{j}{N-1}{\lambda/e}\right]$.
\end{corollary}

\begin{conjecture} \label{cor:del con q hit poly}
Let $\lambda$ be a partition inside an $n\times m$ board and $e$ be an outer corner of $\lambda$, then we have: $P(x;\varnothing) = [m]_n$, and 
\[
P(x;\lambda) = qP(x;\lambda\backslash e) + (xq^m-1)P(x;\lambda/e).
\]
\end{conjecture}

\subsection{Another proof of Lemma~\ref{prop:qhit-relations}}\label{app: another proof qhit rel}
In this section we include 
the proof of~\eqref{eq:keyrel2} using $q$-binomials.

\begin{proof}[Proof of~\eqref{eq:keyrel2} in Lemma~\ref{prop:qhit-relations}]
We use Proposition~\ref{eq: hit in terms of rs} to rewrite both the LHS and RHS in terms of $q$-rook numbers $\qrook{i}{\lambda}$. By $q$-manipulations, we have that~\eqref{eq:keyrel2} is equivalent to
\begin{align}\label{eq:qbinomial}
 \qbinom{m+n-r-i-1}{n-r-1} = q^{(n-r-1)(r-i)} \cdot \left(\sum_{j=r}^{i}\qbinom{m+n-r-j-1}{n-j-1}
 \qbinom{i-r}{j-r}(-1)^{j-r}q^{\binom{j-r}{2}}\right). 
\end{align}
Now, to prove this $q$-binomial identity, we consider 
the following $q$-binomial identities:
\begin{align*}
\prod_{k=0}^{i-r-1} (1-q^k t) &= \sum_{k=0}^{i-r} q^{\binom{k}{2}} \qbinom{i-r}{k}(-1)^k t^k ,   \\
\prod_{k=0}^{m-r-1}\dfrac{1}{(1-q^k t)} &= \sum_{k=0}^{\infty}\qbinom{m-r+k-1}{k}t^k.
\end{align*}
We have that
\begin{multline*}
    \prod_{k=0}^{i-r-1} (1-q^k(q^a t) ) \prod_{k=0}^{m-r-1}\dfrac{1}{(1-q^k(q^b t))}\, = \, \\
    \left( \sum_{j-r=0}^{i-r} q^{\binom{j-r}{2}} \qbinom{i-r}{j-r}(-1)^{j-r} q^{a(j-r)}t^{j-r} \right)
     \cdot \left( \sum_{k=0}^{\infty}\qbinom{m-r+k-1}{k}q^{bk}t^k \right).
\end{multline*}
 Taking the coefficient at $t^{n-r-1}$ at the RHS we get
  \begin{align*}
  &\sum_{j-r=0}^{i-r} q^{\binom{j-r}{2}} \qbinom{i-r}{j-r}\qbinom{m-r+n-j-1}{n-j-1}q^{b(n-j-1)}t^{n-j-1}(-1)^{j-r} q^{a(j-r)}t^{j-r},\\
  &=\sum_{j=r}^i (-1)^{j-r} \qbinom{i-r}{j-r}\qbinom{m-r+n-j-1}{n-j} q^{\binom{j-r}{2}+b(n-j-1) +a(j-r)}.
  \end{align*}
Setting $a=b=r-i$ and denoting by $[t^\ell]@P$ the coefficient of $t^\ell$ at $P$,  we have that
   \begin{align*}
   &\qbinom{m-i+n-r-1}{n-r-1}=[t^{n-r-1}]@ \prod_{k=0}^{m-i-1}\dfrac{1}{(1-q^kt)} \\
   &=  [t^{n-r-1}]@ \prod_{k=0}^{i-r-1} (1-q^k(q^{r-i} t) ) \prod_{k=0}^{m-r-1}\dfrac{1}{(1-q^k (q^{r-i} t))}\\
  &=[t^{n-r-1}]  \sum_{j-r=0}^{i-r} q^{\binom{j}{2}} \qbinom{i-r}{j-r}\qbinom{m-r+n-j-1}{n-j-1}q^{(r-i)(n-j-1)}t^{n-j-1}(-1)^{j-r} q^{(r-i)(j-r)}t^{j-r}\\
  &=\sum_{j=r}^i (-1)^{j-r} \qbinom{i-r}{j-r}\qbinom{m-r+n-j-1}{n-j-1} q^{\binom{j-r}{2} +(r-i)(j-r) + (r-i)(n-j-1)},
  \end{align*}
  as desired.
\end{proof}

\end{document}